\documentclass[a4paper,12pt]{amsart}

\usepackage{latexsym}
\usepackage{amssymb}
\usepackage{graphicx}
\usepackage{wrapfig}
\usepackage{amsthm}
\usepackage{float}
\usepackage{epsfig}
\usepackage{multirow}
\usepackage[toc,page]{appendix}
\usepackage[footnotesize]{caption}
\RequirePackage{color}
\newsavebox{\mybox}

\usepackage{amscd,enumerate}
\usepackage{amsfonts}

\input xy
\xyoption{all}

\usepackage{multicol}
\usepackage{hyperref}
\hypersetup{ colorlinks,
linkcolor=blue,
filecolor=green,
urlcolor=blue,
citecolor=blue }

\newtheorem{lemma}{Lemma}[section]
\newtheorem{corollary}[lemma]{Corollary}

\newtheorem{proposition}[lemma]{Proposition}
\newtheorem{remark}[lemma]{Remark}
\newtheorem{definition}[lemma]{Definition}
\newtheorem{definitions}[lemma]{Definitions}
\newtheorem{example}[lemma]{Example}
\newtheorem{examples}[lemma]{Examples}

\newtheorem{Elemento V.9}[lemma]{Elemento V.9}

\newcommand{\dsum}{\sum}

\newcommand{\K}{{\mathbb{K}}}
\newcommand{\uloopr}[1]{\ar@'{@+{[0,0]+(-4,5)}@+{[0,0]+(0,10)}@+{[0,0] +(4,5)}}^{#1}}
\newcommand{\uloopd}[1]{\ar@'{@+{[0,0]+(5,4)}@+{[0,0]+(10,0)}@+{[0,0]+
(5,-4)}}^{#1}}
\newcommand{\dloopr}[1]{\ar@'{@+{[0,0]+(-4,-5)}@+{[0,0]+(0,-10)}@+{[0,
0]+(4,-5)}}_{#1}}
\newcommand{\dloopd}[1]{\ar@'{@+{[0,0]+(-5,4)}@+{[0,0]+(-10,0)}@+{[0,0
]+(-5,-4)}}_{#1}}
\newcommand{\luloop}[1]{\ar@'{@+{[0,0]+(-8,2)}@+{[0,0]+(-10,10)}@+{[0,
0]+(2,2)}}^{#1}}

\definecolor{turquoise2}{rgb}{0,0.898039,0.933333}
\definecolor{magenta}{rgb}{1,0,1}
\definecolor{olivedrab}{rgb}{0.419608,0.556863,0.137255}
\definecolor{purple2}{rgb}{0.568627,0.172549,0.933333}
\definecolor{amethyst}{rgb}{0.6, 0.4, 0.8}
\definecolor{ao(english)}{rgb}{0.0, 0.5, 0.0}
\definecolor{atomictangerine}{rgb}{1.0, 0.6, 0.4}
\definecolor{amber(sae/ece)}{rgb}{1.0, 0.49, 0.0}
\definecolor{alizarin}{rgb}{0.82, 0.1, 0.26}
\definecolor{auburn}{rgb}{0.43, 0.21, 0.1}
\definecolor{aqua}{rgb}{0.0, 1.0, 1.0}

\begin{document}


\author[Yolanda Cabrera Casado]{Yolanda Cabrera Casado}
\address{Yolanda Cabrera {Casado}:  Departamento de Matem\'atica Aplicada, Universidad de M\'alaga, Campus de Teatinos s/n. 29071 M\'alaga. Spain.}
\email{yolandacc@uma.es}

\author[M\"uge Kanuni]{M\"uge Kanuni}
\address{M\"uge Kanuni: Department of Mathematics. D\"uzce University, Konuralp 81620 D\"uzce, Turkey}
\email{mugekanuni@duzce.edu.tr}

\author[Mercedes Siles ]{Mercedes Siles Molina}
\address{Mercedes Siles Molina: Departamento de \'Algebra Geometr\'{\i}a y Topolog\'{\i}a, Universidad de M\'alaga, Campus de Teatinos s/n. 29071 M\'alaga.   Spain.}
\email{msilesm@uma.es}

\thanks{
The first and the third authors are supported by the Junta de Andaluc\'{\i}a and Fondos FEDER, jointly, through projects  FQM-336 and FQM-7156.
They are also supported by the Spanish Ministerio de Econom\'ia y Competitividad and Fondos FEDER, jointly, through project MTM2016-76327-C3-1-P.
\newline
This research took place while the second author was visiting the Universidad de M\'alaga. She thanks her coauthors for their hospitality.}

\title[Classification of four dimensional perfect non-simple evolution algebras]{Classification of four dimensional perfect non-simple evolution algebras}

\begin{abstract}
We classify  the four dimensional perfect non-simple evolution algebras over a field having characteristic different from 2 and in which there are roots of orders 2, 3 and 7.
\end{abstract}
\maketitle

\section{Introduction and preliminary results}

The aim of this paper is to classify the four dimensional perfect non-simple evolution algebras (an algebra $A$ is said to be perfect if $A^2=A$) under some mild conditions (the base field has characteristic different from 2 and there are roots of orders 2, 3 and 7).

Evolution algebras were introduced by Tian in \cite{Tian}, a pioneering monograph where many connections of these algebras with other mathematical fields (such as graph theory, stochastic processes, group theory, dynamical systems, mathematical physics, etc.) are established.

Since the appearance of this book in 2008, a flurry of papers have arisen. A systematic study of evolution algebras was started in  \cite{CSV1}, considering, not only finite dimensional evolution algebras as  in \cite{Tian}. The paper  \cite{CSV1} analyzes the notions of evolution subalgebras, ideals and non-degeneracy and describe those ideals generated by one element, characterizing the simple evolution algebras. The irreducibility of these algebras was also studied.  The classification of a particular type of algebras is always a major problem, and the case of evolution algebras is not an exception. In the paper \cite{evo9} the two dimensional evolution algebras are classified (on the complex field); over a general field (with mild restrictions), there are 10 mutually non-isomorphic families of evolution algebras (see \cite{CCY}). The three dimensional evolution algebras over a field having characteristic different from 2 and in which there are roots of orders 2, 3 and 7 were classified in \cite{CCY, CSV2}, listing 116 non-isomorphic families. These are precisely the restrictions we carry in. The difficulty in each dimension increases exponentially. There are several reasons: the computations become more complicated because of the number of cases and the properties considered in the previous cases are not sufficient or are not affordable. To visualize this difficulty, note that, roughly speaking, in the classification of three dimensional evolution algebras the matrices that can act as change of basis matrices (when the algebra is perfect) are the permutation matrices: only six, while in the four dimensional case, there are twenty four.

In this paper we start the classification of four dimensional evolution algebras. We will consider  the perfect non-simple.  The simple perfect four dimensional evolution algebras will appear in a forthcoming paper \cite{CKS2}.

The article is divided into five sections, being the first one devoted to the introduction and the preliminaries. In the second one we introduce the notion of basic ideal: an evolution ideal having a natural basis that can be extended to a natural basis of the whole algebra, i.e. having the extension property (see  \cite{CSV1}). These ideals play a fundamental role: on the one hand, the classification relays on the dimension of the maximal basic ideals; on the other hand, if $I$ is a maximal basic ideal of a perfect  evolution algebra $A$ having dimension $n$, then $A/I$ is a basic simple evolution algebra (that is, an algebra having no basic ideals)  whenever the dimension of $I$ is $i$, for $i=1, n-1$, or $n/2$, when $n$ even (Proposition \ref{BasicSimpleQuotient}).

We prove that there exist maximal $i$-basic ideals  (maximal basic ideals of dimension $i$)  for some $i <{\rm dim}(A)$ (Lemma \ref{maximal}), which are not necessarily unique (Example \ref{maximalNotUnique}) and characterize the basic simplicity for perfect evolution algebras, which does not depend on the natural basis.

Important key points for the classification are to write the structure matrix of a perfect non-simple evolution algebra in the form
$\begin{pmatrix}
W & U \\
0& Y
\end{pmatrix}$, where $W$ is the structure matrix of some maximal $i$-basic ideal $I$ and $Y$ is the structure matrix of the evolution algebra $A/I$ (Proposition \ref{simple-mat}), and to show that for  $i=1$, or $n/2$, when $n$ even, the possible permutations producing isomorphic evolution algebras having structure matrices of this form act as permutations in each block (Proposition \ref{OneDimensional}). A consequence is that, in these cases, the number {of} zeros in $W$, $U$ and $Y$ is invariant.

For the classification, {we first consider} if the evolution algebra is irreducible or not. When the algebra is in the first case, we distinguish depending on the dimension of a maximal basic ideal (which always exists) and can be one, two or three. These are the cases considered in Sections 3, 4 and 5, respectively.

In Sections 3 and 4 the philosophy is the same: we fix $W$ and classify depending on the number of zeros in $U$. The matrix $W$ in Section 3 is always the $1 \times 1$ matrix $(1)$, and the number of zeros in $U$ can be 2, 1 or 0. The matrix $W$ in Section 4 corresponds to the structure matrix of a 2-dimensional perfect evolution algebra. Looking at the classification of these evolution algebras (see \cite{CCY}) we distinguish five cases (4.1 {through} 4.5). In each one we classify depending on the number of zeros in $U$, which can be 3, 2, 1 or 0. Once we have {fixed} $W$ and $U$, we discuss which are the matrices $Y$ associated to simple three, respectively two dimensional, evolution algebras that can be considered: three dimensional in Section 3 and two dimensional in Section 4.

The last {section} deals  with the case in which $W$ is the structure matrix of a 3-basic ideal, say $I$. We will consider only the cases in which $A$ is an irreducible evolution algebra (this is explained in Figures \ref{AntesDeThreeWithTwo} and \ref{ThreeWithTwo}).
Subsections 5.1 and 5.2  distinguish when $I$ {does or does not have}
a 2-basic ideal, respectively. Finally, Subsection 5.1 is divided into two cases depending on $A$ having another 3-basic ideal $J$ such that $I \cap J$ is a 2-basic ideal of $I$ and $J$ (in this case {$A$ is said to satisfy} Condition (3,2,3)). We use the associated graph to determine all the algebras satisfying this condition.

In this paper we have not {analyzed} the possible isomorphisms among evolution algebras in the same family; the reason is not to enlarge {the paper}. This will be done in \cite{BCS}.

\medskip

From now on $A$ will be a four dimensional evolution algebra
having characteristic different from 2 and in which there are roots of orders 2, 3 and 7 whenever the structure matrix is invertible.

An \emph{evolution algebra} over a field $\mathbb K$ is a $\mathbb K$-algebra $A$ provided with
a basis $B=\{e_{i} \ \vert \ i\in \Lambda \}$ such that $e_{i}e_{j}=0$
whenever $i\neq j$ (where $\Lambda$ is a non-empty arbitrary set).
Such a basis $B$ is called a \emph{natural basis}.
Fixed a natural basis $B$ in $A,$ the scalars $\omega _{ki}\in \mathbb K$ such that
 $e_{i}^{2}:=e_ie_i=\dsum_{k\in \Lambda} \omega _{ki}e_{k}$ will be called the \emph{structure constants} of $A$\emph{\ relative to} $B$, and the matrix $M_B:= \left(w_{ki}\right)$ is said to be the \emph{structure matrix of} $A$ \emph{relative to} $B$. We will write $M_B(A)$ to emphasize the evolution algebra we refer to. Every evolution algebra is uniquely determined by its structure matrix.


\begin{definition}
\rm
An evolution algebra $A$ is said to be \emph{perfect} if $A^2=A$.
\end{definition}

Recall that an algebra $A$ is said to be \emph{simple} if it has no nonzero proper ideals. For evolution algebras, this is equivalent to not having nonzero proper
evolution ideals, i.e., ideals which are evolution algebras (see \cite{CSV1}).

Let $A$ be a perfect evolution algebra. Then, either $A$ is simple, or it is not. In the second case,  as it will be remarked in Proposition \ref{simple-mat}, there exists a natural basis $B$ of $A$ such that the structure matrix of $A$ relative to $B$ is
$$
M_B=\begin{pmatrix}
W_{m\times m} & U_{m\times (n-m)} \\
0_{(n-m)\times m} & Y_{(n-m)\times (n-m)}
\end{pmatrix}.
$$

\medskip


To reach the classification of four dimensional perfect non-simple evolution algebras, one of the key points will be the graph associated to the evolution algebra. We recall here some of the essential notions for a graph. A \emph{directed graph} is a 4-tuple $E=(E^0, E^1, r_E, s_E)$ consisting of two disjoint sets $E^0$, $E^1$ and two maps
$r_E, s_E: E^1 \to E^0$. The elements of $E^0$ are called the \emph{vertices} of $E$ and the elements of $E^1$ the \emph{edges} of $E$ while for
$f\in E^1$ the vertices $r_E(f)$ and $s_E(f)$ are called the \emph{range} and the \emph{source} of $f$, respectively. If there is no confusion with respect to the graph we are considering, we simply write $r(f)$ and $s(f)$. If $E^0$ and $E^1$ are finite we will
say  that $E$ is \emph{finite}.

A vertex which emits no edges is called a \emph{sink}. A vertex which does not receive any vertex is called a \emph{source}.
A \emph{path} $\mu$ in a graph $E$ is a finite sequence of edges $\mu=f_1\dots f_n$
such that $r(f_i)=s(f_{i+1})$ for $i=1,\dots,n-1$. In this case, $s(\mu):=s(f_1)$ and $r(\mu):=r(f_n)$ are the
\emph{source} and \emph{range} of $\mu$, respectively, and $n$ is the \emph{length} of $\mu$. This fact will be denoted by
$\vert \mu \vert = n$. We also say that
$\mu$ is \emph{a path from $s(f_1)$ to $r(f_n)$} and denote by $\mu^0$ the set of its vertices, i.e.,
$\mu^0:=\{s(f_1),r(f_1),\dots,r(f_n)\}$. On the other hand, by $\mu^1$ we denote the set of edges appearing in $\mu$, i.e., $\mu^1:=\{f_1,\dots, f_n\}$.
We view the elements of $E^{0}$ as paths of length $0$. The set of all paths of a graph $E$ is denoted by ${\rm Path}(E)$.
 Let $\mu = f_1 \dots f_n \in {\rm Path}(E)$.
  If  $n = \vert\mu\vert\geq 1$, and if
$v=s(\mu)=r(\mu)$, then $\mu$ is called a \emph{closed path based at $v$}.
 If
$\mu = f_1 \dots f_n$ is a closed path based at $v$ and $s(f_i)\neq s(f_j)$ for
every $i\neq j$, then $\mu$ is called a \emph{cycle based at} $v$ or simply a \emph{cycle}. A cycle of length $1$ will be said to be a \emph{loop}.

Given a finite graph $E$, the \emph{adjacency matrix} is the matrix $Ad_E=(a_{ij}) \in {\mathbb Z}^{(E^0\times E^0)}$
given by $a_{ij} = \vert\{\text{edges from }i\text{ to }j\}\vert$.

A graph $E$ is said to satisfy  \emph{Condition} (Sing) if among two vertices of $E^0$ there is at most one edge. There are different ways in which  we can associate a graph to an evolution algebra, but for our classification purposes, it will be enough for us to consider only graphs satisfying Condition (Sing), that is, the graph we will consider  will satisfy that $Ad_E=(a_{ij})$ has all its  entries in $\{0, 1\}$. A graph $E$ will be called \emph{cyclic} if there exists a cycle $\mu$ such that $\mu^0=E^0$. The graph $E$ is said to be \emph{strongly connected} if given two different vertices $u, v$, there exists a path $\mu$ such that $s(\mu)=u$ and $r(\mu)=v$.

Given a natural basis $ B=\{e_i\ \vert \ i\in \Lambda\}$ of an evolution algebra $A$ and given its structure matrix $M_B=(\omega_{ji})\in  {\rm M}_\Lambda(\mathbb K)$, consider the matrix
$P^t=(p_{ji})\in  {\rm M}_\Lambda(\mathbb K)$ such that $p_{ji}=0$ if $\omega_{ji}=0$ and $p_{ji}=1$ if $\omega_{ji}\neq 0$.
The \emph{graph associated to the evolution algebra} $A$ (relative to the basis $B$), denoted by $E_A^B$ (or simply by $E$ if the algebra $A$ and the basis $B$ are understood) is the graph whose adjacency matrix is $P= (p_{ij})$ (see, for example, \cite{CSV1}).

\begin{remark}\label{milk}
\rm
In general, the graph associated to an evolution algebra $A$ depends on the selected natural basis  and non-isomorphic graphs can give rise to the same evolution algebra (see \cite[Example 2.34]{CSV1}). However, when the algebra $A$ is finite dimensional and satisfies $A=A^2$, as in the case we are considering, the graph is uniquely determined, as proved in \cite[Corollary 4.5]{EL}. We will denote it by $E_A$. Moreover, when the evolution algebra is non-degenerate (which is the case as perfect implies non-degenerate), the associated graph is connected if and only if the evolution algebra is irreducible (see, for example \cite[Corollary 5.8]{CSV1}).
\end{remark}


\section{Basic ideals}



A notion that will be very useful in order to classify the evolution algebras we are studying is that of basic ideal.

\begin{definitions}
\rm
Let $A$ be an evolution algebra,  $I$ be a nonzero proper ideal of $A$ and  $i\in \{0, 1, 2\dots\}$. We will say that $I$ is an \emph{$i$-basic ideal relative to} a natural basis $B=\{e_j \ \vert \ j \in \Lambda\}$ (where $\Lambda$ is a nonempty set) if $I$ is generated by $i$ vectors from $B$.

We will say that $I$ is a \emph{maximal $i$-basic ideal relative to} $B$ if there are no basic ideals relative to $B$ having dimension $j>i$.

If $A$ has no basic ideals relative to any basis, then we will say that $A$ is \emph{basic simple}.
\end{definitions}

\begin{remark}
\rm
Note that every basic ideal is an evolution ideal and has the extension property (see \cite[Definitions 2.4]{CSV1} for the definitions).
\end{remark}

\begin{lemma}\label{acuarius}
Let $A$ be a perfect finite dimensional evolution algebra and let $I$ be an $i$-basic ideal relative to a natural basis $B$. Then $I$ is an $i$-basic ideal relative to any natural basis of $A$.
\end{lemma}
\begin{proof}
Since $A^2=A$ then the only natural basis that $A$ has consists of a permutation of the elements of $B$ (by \cite[Theorem 4.4]{EL}).
\end{proof}

Lemma \ref{acuarius} gives rise to the following definitions.

\begin{definitions}\label{defbasic}
\rm
Let $A$ be a perfect finite dimensional evolution algebra and let $I$ be an ideal of $A$.

\begin{enumerate}[\rm (i)]
\item $I$ is an \emph{$i$-basic ideal} if it is an $i$-basic ideal relative to any natural basis.
\item $I$ is a \emph{maximal $i$-basic ideal} if it is a maximal $i$-basic ideal relative to any natural basis.
\item $A$ is \emph{basic simple} if the ideal generated by every element in any natural basis is $A$.
\end{enumerate}
\end{definitions}

\begin{lemma}\label{maximal}
Let $A$ be a perfect finite dimensional evolution algebra. Then $A$ contains a maximal $i$-basic ideal for some $i\in \{0,1, \dots, n-1\}$.
\end{lemma}
\begin{proof}
As follows by Lemma \ref{acuarius}, the existence of maximal basic ideals does not depend on the natural basis. Let $B$ be any natural basis of $A$. Since $A$ is perfect, $e^2\neq 0$ for any $e\in B$. If the ideal generated by every element of $B$ is $A$, then $0$ is a maximal basic ideal. Otherwise, let $i\neq {\rm dim}(A)$ be the maximum of the dimensions of the ideals generated by elements in the basis $B$. Then $A$ has a maximal $i$-basic ideal.
\end{proof}

{The example that follows shows that maximal basic ideals are not, necessarily unique. However, there are some cases in which the unicity is true.}

\begin{example}\label{maximalNotUnique}
\rm
A maximal basic ideal is not necessarily unique. For an example, consider the evolution algebra $A$ having a natural basis $B= \{e_1, e_2, e_3, e_4\}$ and with product given by $e_1^2 = e_2$, $e_2^2=e_1$, $e_3^2=e_1+e_2+e_3$, $e_3^2=e_2+e_4$.
and take $I$ and $J$ the ideals generated by $\{e_1, e_2, e_3\}$ and $\{e_1, e_2, e_4\}$, respectively, Then $I$ and $J$ are maximal $3$-basic ideals.

Moreover, the uniqueness does not depend on the irreducibility of the algebra, as the evolution algebra $A$ we have considered in this example is irreducible.
\end{example}

Now we relate the notions of simple and basic simple evolution algebra.

\begin{lemma}\label{aguita}
Let $A$ be a finite dimensional evolution algebra $A$. The following are equivalent conditions:
\begin{enumerate}[\rm (i)]
\item $A$ is simple.
\item $A=A^2$ and $E_A$ is strongly connected.
\item $A=A^2$ and $A$ is basic simple.
\end{enumerate}
\end{lemma}
\begin{proof}
(i) $\Rightarrow$ (ii). If $A$ is simple, then $A=A^2\neq 0$ and the graph $E_A$ is unique (see \cite[Corollary 4.5]{EL} ). Moreover, $E_A$ has to be strongly connected because  if there are two vertices which are not connected, then there exists a proper nonzero ideal.

(i) $\Leftrightarrow$ (iii) follows by \cite[Corollary 4.6]{CSV2}.

(ii) $\Rightarrow$ (iii). Since $E_A$ is unique (because $A$ is perfect) and a nonzero proper basic ideal would give rise to a non strongly connected graph, the result follows.
\end{proof}

We cannot eliminate the hypothesis $A=A^2$ in Lemma \ref{aguita}, as shown in the example that follows.

\begin{example}
\rm
Consider the evolution algebra $A$ having a natural basis $\{e_1, e_2\}$ such that $e_i^2=e_1+e_2$, for $i=1, 2$. This evolution algebra is neither simple nor basic simple as $e_1+e_2$ generates a one dimensional ideal (which is a basic ideal). However, one associated graph for $A$ is

$$ 
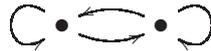
\xymatrix{
{\bullet}  \ar@(ul,dl) \ar@/^-.5pc/[r]   & {\bullet}
\ar@(ur,dr)   \ar@/_.5pc/[l]}
$$
which is strongly connected.

\end{example}

\begin{corollary}
 For a perfect evolution algebra, being basic simple does not depend on the natural basis.
\end{corollary}

From now on, if $S$ is a subset of an algebra $A$ then we will denote by
$\left\langle S\right\rangle $ the ideal of $A$ generated by $S.$

Our classification relies on the result that follows.

\begin{proposition}\label{BasicSimpleQuotient}
Let $A$ be a finite $n$-dimensional evolution perfect algebra (for $n\geq 2$) and let $I$ be a maximal $i$-basic ideal having dimension 1, $n-1$ or $n/2$ (being $n$ even in this last case).
Then $A/I$ is a basic simple perfect evolution algebra.
\end{proposition}
\begin{proof}
Let $B=\{e_j \ \vert \ j \in \{1, \dots, n\} \}$ be a natural basis of $A$ such that $I=\langle \{e_1, \dots, e_i\}\rangle$.
It is easy to see that $A/I$ is an evolution algebra with $B':=\{\overline{e_{i+1}}, \dots, \overline{e_n}\}$ being a natural basis.

We see first that $A/I$ is a perfect evolution algebra.

{\bf Case $i=1$}. Write
$M_B =
\begin{pmatrix}
\alpha & U \\
0_{n-1 \times 1} & Y
\end{pmatrix}$, where $\alpha\in K$, $U\in M_{1 \times n-1}(K)$ and $Y=M_{n-1}(K)$. Then $\vert M_B \vert = (-1)^n\alpha \vert Y \vert$. Since $A$ is perfect, then $\vert M_B \vert \neq 0$; therefore $\alpha \neq 0$ and $\vert Y \vert\neq 0$.

Note that $A/I$ is an $(n-1)$-dimensional evolution algebra and $Y=M_{B'}$. As we have proved, $\vert Y \vert\neq 0$, which implies that $A/I$ is perfect.

{\bf Case $i=n-1$}. Write
$M_B =
\begin{pmatrix}
W & U \\
0_{1\times n-1} & \alpha
\end{pmatrix}$, where $W\in M_{n-1}(K)$, $U\in M_{n-1 \times 1}(K)$ and $\alpha\in K$. Then $\vert M_B \vert = (-1)^n \alpha \vert W \vert$. Since $A$ is perfect,  $\vert M_B \vert \neq 0$; therefore $\vert W \vert\neq 0$ and $\alpha \neq 0$.

Note that $A/I$ is a $1$-dimensional evolution algebra which is perfect as $\alpha \neq 0$.

{\bf Case $i=n/2$ and $n$ is even.} Write
$M_B =
\begin{pmatrix}
W & U \\
0_{n/2} & Y
\end{pmatrix}$, where $W, U, Y\in M_{n/2}(K)$. By  \cite[Theorem 2]{Silvester} we have $\vert M_B \vert =  \vert W \vert  \vert Y \vert$. Since $A$ is perfect, $\vert M_B \vert \neq 0$; therefore $\vert W \vert\neq 0$ and $\vert Y \vert\neq 0$.

Reasoning as in the first case we obtain that $A/I$ is perfect.

Now we prove that $A/I$ is basic simple. If this is not the case, then there exists a $j$-basic ideal $\overline J$ of $A/I$, for some $j$. Since $A/I$ is perfect (as we have proved before), by Lemma \ref{acuarius} the basic ideals of $A/I$ do not depend on the basis.
 This means that we may assume that there exists $\overline{B_J}= \{\overline{e_k} \ \vert \ k\in \Lambda_j \ \}$, where $\Lambda_j\subsetneq \{i+1, \dots, n\}$, such that $\overline J$ is generated by $\overline{B_J}$. Then $\{e_1, \dots, e_i\}\cup \{e_k\}_{k\in \Lambda_j}$  generates an ideal $I'$ of $A$ having dimension $i+j\neq n$. Note that $I'$ is an ($i+j$)-basic ideal, a contradiction to the maximality of $i$ as the dimension of a basic ideal.
\end{proof}

\medskip

\begin{definition}
\rm
Let $A$ be an evolution algebra and assume that there exists a natural basis $B$ such that
$M_B =
\begin{pmatrix}
W & U \\
0& Y
\end{pmatrix}$.
We will say that the \emph{number of zeros} of $W$ \emph{is invariant} if for any basis $B'$ such that $M_{B'} =
\begin{pmatrix}
W' & U' \\
0& Y'
\end{pmatrix}$,
the number of zeros in $W'$ coincides with the number of zeros in $W$. Analogous definitions can be given for $U$ and $Y$.
\end{definition}

\begin{examples}\label{EjemplosDeTodo}
\rm
Let $A$ be an evolution algebra and assume that there exists a natural basis $B$ such that
$M_B =
\begin{pmatrix}
W & U \\
0& Y
\end{pmatrix}$.
\begin{enumerate}[\rm (i)]
\item\label{Ej1} The number of zeros of $U$ and of $W$ is not necessarily invariant. Take $A$ as the evolution algebra with natural basis $B=\{e_1, e_2, e_3, e_4\}$ and product given by

\medskip

\begin{center}
$M_B$:=\scalebox{0.60}{$\left(\begin{tabular}{ccc|c}
0 & 1 &  1  & 0\\
1 & 0& 1 & 1\\
0 & 0 &  1  & 0\\
\hline
0 & 0 &  0 & 1\\
\end{tabular}\right).
$}
\end{center}

\medskip

Now, take $B':=\{e_4, e_1, e_2, e_3\}$. Then

\medskip

\begin{center}
$M_{B'}$:=\scalebox{0.60}{$\left(\begin{tabular}{ccc|c}
1& 0 &  0  & 0\\
0 & 0& 1 & 1\\
1 & 1 &  0  & 1\\
\hline
0 & 0 &  0 & 1\\
\end{tabular}\right),
$}
\end{center}

\medskip

\noindent
and the reader can see that the statement is true.


\item \label{Ej2}  The number of zeros in $Y$ is not necessarily invariant. Consider the evolution algebra with  natural basis $B=\{e_1, e_2, e_3, e_4\}$, and let $B':=\{e_3, e_1, e_2, e_4\}$. then

\begin{center}
$M_B$:=\scalebox{0.60}{$
\left(\begin{tabular}{c|ccc}
1 & 0 &  0 & 1\\
\hline
0 & 1& 0 & 1\\
0 & 1 &  1 & 1\\
0 & 1 &  0& -1\\
\end{tabular}\right)$} and $M_{B'}$:=\scalebox{0.60}{$
\left(\begin{tabular}{c|ccc}
1 & 0 &  1& 1\\
\hline
0 & 1& 0 & 1\\
0 & 0 &  1 & 1\\
0 & 0 &  1& -1\\
\end{tabular}\right)$},
\end{center}

\medskip

\noindent
which proves the statement.
\item Even if the number of zeros of $U$ is invariant, the number of zeros of $W$ and $Y$ need not  be invariant. Consider the evolution algebra with natural basis $B=\{e_1, e_2, e_3, e_4\}$ and product given by

\begin{center}
	$M_B$:=\scalebox{0.60}{$
\left(\begin{tabular}{cc|cc}
1 & 1 &  0 & 0\\
0 & 1& 0 & 1\\
\hline
0 & 0 &  1 & 0\\
0 & 0 &  0& 1\\
\end{tabular}\right)$}.
\end{center}

Then, for every permutation $\sigma\in S_4$ such that $M_{B^\sigma}:=
\begin{pmatrix}
W^\sigma & U^\sigma\\
0 & Y^\sigma \\
\end{pmatrix},$
where $B^\sigma:=\{e_{\sigma(1)}, e_{\sigma(2)}, e_{\sigma(3)}, e_{\sigma(4)}\}$, we have that the number of zeros of $U^\sigma$ is three. This means that the number of zeros of $U$ is invariant.
Now, take $B':=\{e_3, e_1, e_4, e_2\}$. Then

\begin{center}
	$M_B'$:=\scalebox{0.60}{$
\left(\begin{tabular}{cc|cc}
1 & 0 &  0 & 0\\
0 & 1& 0 & 1\\
\hline
0 & 0 &  1 & 0\\
0 & 0 &  1& 1\\
\end{tabular}\right)$},

\end{center}

\noindent
what shows that the number of zeros of $W$ and $Y$ is not invariant.
\item \label{Ej3}  Even if the number of zeros of $W$ is  invariant, the number of zeros of $U$ and $Y$ need not  be. Take $B$ and $B'$ as in \eqref{Ej2}.
\item\label{Ej4} Even if the number of zeros of $Y$ is  invariant, the number of zeros of $W$ and $U$ need not  be. Take $B$ and $B'$ as in \eqref{Ej1}.
\end{enumerate}
\end{examples}

We will obtain the classification of evolution algebras in this case taking into account the
number of nonzero entries and the possible expression of the structure matrix, as explained in \cite[Corollary 4.6]{CSV1}. For the sake of completeness, we include it here.

\begin{proposition}\label{simple-mat}
Let $A$ be a finite-dimensional evolution algebra of dimension $n$ and $B= \{e_i \ \vert \ i \in \Lambda\}$ a natural basis of $A$. Then $A$ is simple if and only if the determinant of the structure matrix $M_B(A)$ is nonzero and $B$ cannot be reordered in such a way that the corresponding structure matrix is as follows:
\begin{equation*}
 \left(
\begin{array}{cc}
W_{m\times m} & U_{m\times (n-m)} \\
0_{(n-m)\times m} & Y_{(n-m)\times (n-m)}
\end{array}
\right),
\end{equation*}
for some $m\in \mathbb N$ with $m<n$ and matrices $W_{m\times m},$\textit{\ }$U_{m\times
(n-m)}$\textit{\ and }$Y_{(n-m)\times (n-m)}.$
\end{proposition}

Following the same notation as in \cite[Subsection 3.1]{CSV2}, for any natural number $n$, we define the semidirect product $S_n \rtimes (\K^{\times})^n$.
It is not difficult to see that any matrix $P$ in $S_n \rtimes (\K^{\times})^n$ is a change of basis matrix from a natural basis $B$ into another natural basis $B'$ and the relationship among the structure matrices $M_B$ and $M_{B'}$ and the matrix $P$ is as given in \cite[Condition (5)]{CSV2}, that is, $P^{-1}M_BP^{(2)}=M_{B'}$. This is the reason because we defined in \cite[Subsection 3.1]{CSV2} the action of $P$ on $M_B$ by:
$$P \cdot M_B = P^{-1}M_BP^{(2)}.$$

For any $\sigma \in S_n$, we denote by $I_\sigma$ the matrix obtained from the identity  changing the columns as determined by $\sigma$ (i.e. the $i$-column in $I_\sigma$ is $C_{\sigma(i)}$, where $C_j$ denotes the $j$-column in the identity matrix).

The result that follows is a generalization to $n$-dimensional evolution algebras of \cite[Proposition 3.2]{CSV2}. We do not include its proof as it is similar.

\begin{proposition}\label{BatidoDeFresasConLeche}
For any natural number $n$, any $P\in S_n\rtimes (\K^\times)^n$ and any $M \in \mathcal{M}_n(\K)$ we have:
\begin{enumerate}[\rm (i)]
\item\label{numerodeceros1} The number of zero entries in $M$ coincides with the number of zero entries in $P \cdot M$.
\item\label{numerodeceros2} The number of zero entries in the main diagonal of $M$ coincides with the number of zero entries in the main diagonal of $P \cdot M$.
\item\label{numerodeceros3} The rank of $M$ and the rank of $P\cdot M$ coincide.
\item\label{numerodeceros4} Assume that $M$ is the structure matrix of an evolution algebra $A$ relative to a natural basis $B$. Assume that $A^2=A$. If $N$ is the structure matrix of $A$ relative to a natural basis $B'$ then there exists $Q\in S_n\rtimes (\K^\times)^n$ such that $N= Q \cdot M$.
\end{enumerate}
\end{proposition}


In order to study the change of basis matrices producing isomorphic evolution algebras whose structure matrix is the same as the inicial one, we stablish Proposition \ref{OneDimensional} . First we need some definitions.
Recall that a \emph{reducible evolution algebra} is an evolution algebra $A$
which can be decomposed as the direct sum (in the sense of \cite[Definition 5.3]{CSV1}) of two nonzero evolution algebras, equivalently, of two nonzero evolution ideals, equivalently, of two nonzero ideals, as shown in \cite[Lemma 5.2]{CSV1}. An evolution algebra which is not reducible will be called \emph{irreducible}.

The following definitions are \cite[Definitions 3.1]{CSV1}.

Let $B=\{e_{i}\ \vert \ i\in \Lambda \}$ be a natural basis
of an evolution algebra $A${\ and let }$i_{0}\in \Lambda .${\ The
\emph{first-generation descendents} of }$\ i_{0}$ {\ are the elements
of the subset }$D^{1}(i_{0})${\ given by: }
\begin{equation*}
D^{1}(i_{0}):=\left\{k\in \Lambda \ \vert \ e_{i_{0}}^{2}=\sum_k \omega _{ki_{0}}e_{k}\text{
with }\omega _{ki_{0}}\neq 0\right\}.
\end{equation*}
{In an abbreviated form, }$D^{1}(i_{0}):=\{j\in \Lambda\ \vert\ \omega
_{ji_{0}}\neq 0\}.${\ Note that }$j\in D^{1}(i_{0})${\ if and only
if, }$\pi _{j}(e_{i_{0}}^{2})\neq 0${\ (where }$\pi _{j}${\ is the canonical
projection of }$A${\ over }$\mathbb Ke_{j}${).}

{Similarly, we say that }$j${\ is a \emph{second-generation
descendent} of }$i_{0}${\ whenever }$j\in D^{1}(k)${\ for some }$
k\in D^{1}(i_{0}).${\ Therefore,}
\begin{equation*}
D^{2}(i_{0})=\bigcup\limits_{k\in D^{1}(i_{0})}D^{1}(k).
\end{equation*}
{By recurrency, we define the set of \emph{mth-generation descendents
} of }$i_{0}${\ as}
\begin{equation*}
{\ }D^{m}(i_{0})=\bigcup\limits_{k\in D^{m-1}(i_{0})}D^{1}(k).
\end{equation*}
{\ Finally, the \emph{set of descendents} of }$i_{0}${\ is
defined as the subset of }$\Lambda ${\ given by }
\begin{equation*}
D(i_{0})=\bigcup\limits_{m\in \mathbb{N}}D^{m}(i_{0}).
\end{equation*}
On the other hand, we say that $j\in \Lambda$ is an
\emph{ascendent} of $i_{0}$ if $i_{0}\in D(j);$ that is, $i_{0}$
is a \emph{descendent of} $j.$
\medskip

\begin{proposition}\label{OneDimensional}
Let $A$ be an $n$-dimensional perfect evolution algebra (for $n\geq 3$), and assume that there exists a maximal $s$-basic ideal $I$, for $s\in \{1, n/2\}$ (being $n$ even in this last case).
This implies that there exists
a natural basis $B= \{e_1, \dots, e_n \}$ such that $I=\langle \{e_1, \dots, e_s\}\rangle$ and
$M_B =
\begin{pmatrix}
W & U \\
0& Y
\end{pmatrix}$, where $W \in M_{s}(K)$ (in fact, $W$ is the structure matrix of $I$ relative to the basis $\{e_1, \dots, e_s\}$ of $I$), $U\in M_{s \times (n-s)}(K)$, $Y\in  M_{n-s}(K)$.
\begin{enumerate}[\rm (i)]
\item For $s=1$, the algebra $A$ is irreducible, the only possible permutations $\sigma$ producing isomorphic evolution algebras whose structure matrix has the same form as $M_B$ must satisfy  $\sigma(1)=1$.
\item For $s=1$, the number of zeros  in $W$, $U$ and $Y$ is  invariant.
\item Assume $s=n/2$ and $A$ irreducible. Then, the only possible permutations $\sigma$ producing isomorphic evolution algebras whose structure matrix has the same form as $M_B$ must satisfy  $\sigma(i)\in \{1, \dots, n/2\}$ for all $i \in \{1, \dots, n/2\}$ (and, consequently, $\sigma(j)\in \{(n/2)+1, \dots, n\}$ for all $j \in \{(n/2)+1, \dots, n\}$).
\item For $s=n/2$ and $A$ irreducible, the number of zeros in $W$, $U$ and $Y$ is  invariant. If $A$ is not irreducible, then the result is not true.
\end{enumerate}
\end{proposition}
\begin{proof}
 Write $M_B=(\omega_{ji})$, where $\omega_{ji}\in K$.

(i) and (ii). Assume first $s=1$. It is trivial that $A$ has to be irreducible because otherwise there will be a maximal $t$-basic ideal of dimension $1  < t < n-1$, which is not the case. It is immediate to see that the number of zeros of $W$ is invariant. Now we see that the number of zeros in $Y$ is invariant, which consequently would imply that the number of zeros in $U$ is also invariant by (i) in Proposition \ref{BatidoDeFresasConLeche}.

Assume that we may change $e_1$ to $e_i$, for some $i\in \{2, \dots, n\}$. Then $e_i^2=\omega_{ii}e_i$ and this would imply that $\{e_1, e_i\}$ generates a $2$-basic ideal, which is not possible as $I$ is a maximal $1$-basic ideal. This means that any other change is given by a permutation $\sigma\in S_n$ such that $\sigma (e_1)=e_1$. Therefore, the first row of the matrix
$I_\sigma M_B$ is just  $  \begin{pmatrix}
\omega_{11} & \omega_{1\sigma(2)} &  \dots &  \omega_{1\sigma(n)}\\
\end{pmatrix}$,
which has the same number of zeros that the first row of the matrix $M_B$, which coincides with the number of zeros of $U$.

(iii) and (iv). Assume $s=n/2$ and $A$ irreducible.
%
%
%
%
Note that $Y$ is the structure matrix of the evolution algebra $A/I$ relative to the basis
$\{\overline{e_{(n/2)+1}}, \dots, \overline{e_n}\}$. By Proposition \ref{BasicSimpleQuotient}, the evolution algebra $A/I$ is perfect and basic simple.
Decompose $\{1, \dots, n\}= \Lambda_1 \sqcup \Lambda_2$, where $\Lambda_1:=\{1, \dots, n/2\}$ and $\Lambda_2:=\{(n/2)+1, \dots, n\}$.
We are going to prove that the only bases $B'=\{e_{\sigma(1)}, \dots e_{\sigma(n)}\}$, for $\sigma \in S_n$, such that $M_{B'}=
\begin{pmatrix}
W' & U' \\
0& Y'
\end{pmatrix}$, where $W' \in M_{n/2}(K)$, are those such that $\sigma(\Lambda_i)=\Lambda_i$, for $i=1, 2$. Once this is proved, we will have that the number of zeros in $W$, in $Y$ and in $U$ is invariant.
Indeed, let $\sigma\in S_n$ be such that $\sigma(i) =j$, for some $i\in \Lambda_1$ and $j\in \Lambda_2$. Then, for every $k\in D^1(i)$, necessarily $\sigma(k)\in \Lambda_1$; this implies $\sigma(D(i))\subseteq \Lambda_1$. Since $A/I$ is basic simple, then $D(i)=\Lambda_2$, by \cite[Corollary 4.10]{CSV1}, so $\sigma(\Lambda_2)\subseteq \Lambda_1$, thus $\sigma(\Lambda_2)=\Lambda_1$ (because both have the same number of elements). Taking into account that $B^\sigma:=\{e_{\sigma(l)} \ \vert \ l\in\{1, \dots, n\}\}$
is such that $M_{B^\sigma}$ has the same form as the matrix $M_{B'}$ before, we get that the vector space generated by $\{e_l \ \vert\ l \in \Lambda_2\}$, say $J$, is an ideal of $A$. This implies that $A= I \oplus J$, where $I$ is the ideal of $A$ generated by $\{e_l \ \vert\ l \in \Lambda_1\}$, a contradiction as we are assuming that $A$ is irreducible.

Finally, we see that the result in (ii) is not true when the algebra is not irreducible. For an example, consider the evolution algebra $A$ having a natural basis $B=\{e_1, e_2, e_3, e_4\}$ such that

\begin{center}
$M_B=$\scalebox{0.60}{$
\left(\begin{tabular}{cc|cc}
1 & 1 &  0 & 0\\
-1 & 1& 0 & 0\\
\hline
0 & 0 &  0 & 1\\
0 & 0 &  1& 0\\
\end{tabular}\right)$}.
\end{center}

Then for $B'=\{e_3, e_4, e_1, e_2\}$ we have

\begin{center}
	$M_B'$=\scalebox{0.60}{$
\left(\begin{tabular}{cc|cc}
0 & 1 &  0 & 0\\
1 & 0& 0 & 0\\
\hline
0 & 0 &  1 & 1\\
0 & 0 &  -1& 1\\
\end{tabular}\right),$}
\end{center}

\noindent
and we observe that the number of zeros in the pieces is not invariant.
\end{proof}

\begin{corollary}\label{UnicidadMaximal}
Let $A$ be an $n$-dimensional perfect evolution algebra (for $n\geq 3$), and assume that there exists a maximal $s$-basic ideal $I$, for $s\in \{1, n/2\}$ (being $n$ even in this last case). Then $I$ is the unique maximal basic ideal.
\end{corollary}
\begin{proof}
For $s=1$, if $A$ has two maximal 1-basic ideals, then the sum of both gives a 2-basic ideal of $A$, contradicting the maximality of each one.

Now, assume $s=n/2$ and use the same notation as in the proof of Proposition \ref{OneDimensional} (ii). Let $I=\{e_1, \dots, e_{n/2}\}$ be a maximal $n/2$-basic ideal. Assume that $J$ is another maximal $n/2$-basic ideal of $A$. This means that a basis for $J$ is $\{e_{\sigma(1)}, \dots, e_{\sigma(n/2)}\}$, where $\sigma\in S_n$. By the proof of (ii), $\{\sigma(1), \dots, \sigma (n/2)\} = \{1, \dots, n/2\}$. This shows  $I=J$.
\end{proof}

\begin{remark}
\rm
\begin{enumerate}[\rm (i)]
\item The integer $s$  in Proposition \ref{OneDimensional} cannot be $n-1$. For an example, consider $n=4$, $s=3$ and take $A$ as in Examples \ref{EjemplosDeTodo} (i).
\item The integer $s$ in  Proposition \ref{OneDimensional} cannot be 0 if $A$ is not perfect (i.e. if $A$ is not perfect and $0$ is a maximal ideal, then the number of zeros is not an invariant). For an example, consider the evolution algebra $A$ having a natural basis $B=\{e_1, e_2, e_3\}$ and product given by 	$M_B$=\scalebox{0.60}{$\begin{pmatrix}
0 & 1 &  1 \\
1 & 0& 0 \\
0 & 1 &  1 \\
\end{pmatrix}$}. Then for $B'=\{e_1, e_2+e_3, -e_2+e_3\}$ we have
$M_B'$=\scalebox{0.60}{$\begin{pmatrix}
0 & 2 &  2 \\
1/2 & 1& 1 \\
-1/2 & 1 &  1 \\
\end{pmatrix}$}. Note that while $M_B$ has four zeros, $M_{B'}$ has only one.
\end{enumerate}
\end{remark}

\begin{example}\label{JabonDeMuge}
\rm
The evolution algebras having structure matrices \
$M_B$:=\scalebox{0.60}{$\begin{pmatrix}
1 & 1 &  0 & 0\\
0 & 1& 0 & 1\\
0 & 0 &  1 & 0\\
0 & 0 &  0& 1\\
\end{pmatrix}$} \ and\
$M_{B'}$:=\scalebox{0.60}{$\begin{pmatrix}
1 & 0 &  0 & 0\\
0 & 1& 0 & 1\\
0 & 0 &  1 & 0\\
0 & 0 &  1& 1\\
\end{pmatrix}$}
are isomorphic. In fact, $I_{(1,2,4,3)} \cdot M= M'.$
Note that the corresponding evolution algebra is reducible: if $B=\{v_1, v_2, v_3, v_4\}$, then the evolution algebra  with basis $B$ can be written as
$A= I \oplus J$, where $I=\langle\{v_1, v_2, v_4\}\rangle$ and $J=\langle \{v_3\} \rangle$.
\end{example}

\begin{example}
\rm
Even if the evolution algebra is  irreducible, the number of zeros in $W, U$ and $Y$ is not, necessarily invariant, as  (ii) in Examples \ref{EjemplosDeTodo}  shows.
\end{example}
\medskip

In order to classify the non-simple four dimensional perfect algebras we will distinguish two cases depending on whether or not the algebra is reducible or not.

Assume first that the evolution algebra is reducible.
Assume we have  $A= I\oplus J$, where $I$ and $J$ are evolution ideals having dimensions 1 and 3, or having both of them dimension 2. When this happens, the classification is achieved by considering the classification of two dimensional evolution algebras \cite{evo9} and the classification of three dimensional evolution algebras (see \cite{CSV2}).

\medskip

Now, assume that the evolution algebra is  irreducible.
We classify the non-simple irreducible perfect four-dimensional evolution algebras $A$ taking into account the dimension of the maximal basic ideal (that exists and is unique by Lemma \ref{maximal}), call it $I$. Since $A$ is non-simple, the dimension of this maximal basic ideal, denote it by $i$, is at least one.

Let $B=\{e_1, e_2, e_3, e_4 \}$ be a natural basis of $A$, write $\{1, 2, 3, 4\}= \Lambda_1 \sqcup \Lambda_2$, where $I$ is the vector space generated by $\{e_j\}_{j \in \Lambda_1}$. There is no loss in generality if we assume that $\Lambda_1= \{1, \dots, i\}$, and let

\begin{equation}\label{inicios}
M_B= \left(
\begin{array}{cc}
W& U\\
0 & Y
\end{array}
\right),
\end{equation}
where $W$ is the structure matrix of $I$ relative to the natural basis $\{e_1, \dots, e_i\}$.
By Proposition \ref{BasicSimpleQuotient}, $A/I$ is a basic simple evolution algebra and $Y$ can be seen as the structure matrix of $A/I$ relative to the natural basis $\{\overline{e_{i+1}}, \dots, \overline{e_n}\}$.


\medskip

Note that in Sections \ref{Case2.1} and \ref{Case2.2}, the maximal basic ideal is unique by Corollary \ref{UnicidadMaximal}. This is not the case in Section \ref{SectionThreeDimensional}.

\section{The maximal basic ideal is one-dimensional}\label{Case2.1}

Assume that  the maximal basic ideal $I$ of $A$ is one-dimensional. Write
\begin{equation}\label{churritos}M_B =\scalebox{0.60}{$\begin{pmatrix}
1 & \omega_{12} &  \omega_{13} &  \omega_{14}\\
0 & \omega_{22} &  \omega_{23} &  \omega_{24}\\
0 & \omega_{32} &  \omega_{33} &  \omega_{34}\\
0 & \omega_{42} &  \omega_{43} &  \omega_{44}\\
\end{pmatrix}$}.
\end{equation}
Note that, in this case,
\begin{equation*}\label{UeY}
U= \scalebox{0.60}{$\begin{pmatrix}
 \omega_{12} &  \omega_{13} &  \omega_{14}
\end{pmatrix}$},
\quad
Y= \scalebox{0.60}{$\begin{pmatrix}
 \omega_{22} &  \omega_{23} &  \omega_{24}\\
\omega_{32} &  \omega_{33} &  \omega_{34}\\
\omega_{42} &  \omega_{43} &  \omega_{44}\\
\end{pmatrix}$}.
\end{equation*}

By Proposition \ref{OneDimensional}, the number of zeros in $U$ and $Y$ is invariant. We will use this fact to classify. Concretely, the cases we are going to consider depends on the number of zeros in $U$.
Notice that $\omega_{12}$, $\omega_{13}$ and $\omega_{14}$ cannot be zero at the same time because otherwise the algebra will be reducible.
\medskip

\subsection{The matrix $U$ has two zero entries.}\label{Case2.1.1}

The possible matrices are of the form:
\medskip

\begin{center}
\scalebox{0.60}{$\begin{pmatrix}
1 & 0 & 0 &  \omega_{14}\\
0 & \omega_{22} &  \omega_{23} &  \omega_{24}\\
0 & \omega_{32} &  \omega_{33} &  \omega_{34}\\
0 & \omega_{42} &  \omega_{43} &  \omega_{44}\\
\end{pmatrix},
\quad
\begin{pmatrix}
1 & 0 &  \omega_{13} & 0\\
0 & \omega_{22} &  \omega_{23} &  \omega_{24}\\
0 & \omega_{32} &  \omega_{33} &  \omega_{34}\\
0 & \omega_{42} &  \omega_{43} &  \omega_{44}\\
\end{pmatrix} ,
\quad
\begin{pmatrix}
1 & \omega_{12} & 0 &  0\\
0 & \omega_{22} &  \omega_{23} &  \omega_{24}\\
0 & \omega_{32} &  \omega_{33} &  \omega_{34}\\
0 & \omega_{42} &  \omega_{43} &  \omega_{44}\\
\end{pmatrix},
$}
\end{center}
\medskip

\noindent
where the entries in the first row are nonzero. It happens that the three of them produce isomorphic evolution algebras and the change of basis matrices are given by $I_{(3,4)}$ (producing the isomorphism between the first and the second) and by $I_{(2,4)}$ (producing the isomorphism between the first and the third). Moreover, by multiplying conveniently one of the vectors in the basis by an scalar, we may assume that the matrix is:

\begin{equation}\label{churritoUno}\scalebox{0.60}{$
\begin{pmatrix}
1 & 1 & 0 &  0\\
0 & \omega_{22} &  \omega_{23} &  \omega_{24}\\
0 & \omega_{32} &  \omega_{33} &  \omega_{34}\\
0 & \omega_{42} &  \omega_{43} &  \omega_{44}\\
\end{pmatrix}$.}
\end{equation}

In order to get an irredundant classification we may keep the first row and the first columns as in \eqref{churritoUno} (use what has been explained above and use Proposition \ref{OneDimensional}). This implies that the only possible change of basis which is allowed is $I_{(3,4)}$. Moreover, the matrix $Y$ corresponds to a basic simple three-dimensional evolution algebra. Since $\vert M_B\vert \neq 0$, $\vert Y \vert \neq 0$ and by Lemma \ref{aguita}, this is equivalent to saying that  $Y$ is the structure matrix of a simple three-dimensional evolution algebra. Therefore, in the tables that follow we have inserted, as matrix $Y$, all the different matrices associated to simple three-dimensional evolution algebras, which appear in the classification in  \cite{CCY, CSV2}.

Now, we explain how do we select the matrices in the first columns using the classification of the three-dimensional evolution algebras. Take as a reference, for example, Table 19 in \cite{CCY} and look at the fifth row (corresponding to the first simple three-dimensional evolution algebras in that table), which is the following:

\begin{center}
\scalebox{0.60}{
\begin{tabular}{|c||c||c||c||c||c|}

\hline
&&&&&\\
&(1,2)&(1,3)&(2,3)&(1,2,3)&(1,3,2)\\
&&&&&\\[-0.2cm]

&&&&&\\[-0.2cm]
\hline
\hline
&&&&&\\[-0.2cm]

${
\begin{pmatrix}
\mu & 0 & 1 \\
1 & 0 & 0 \\
0  & 1 & 0

\end{pmatrix}}
$  &

${
 \begin{pmatrix}
0 & 1 & 0 \\
0 & \mu & 1 \\
1  & 0 & 0

\end{pmatrix}}$
&
${
\begin{pmatrix}
0 & 1 & 0 \\
0 & 0 & 1 \\
1  & 0 & \mu

\end{pmatrix} }$
 &

${
\begin{pmatrix}
\mu & 1 & 0 \\
0 & 0 & 1 \\
1  & 0 & 0
\end{pmatrix}}$
&
${
\begin{pmatrix}
0 & 0 & 1 \\
1 & \mu & 0 \\
0  & 1 & 0
\end{pmatrix}}$
&
${
\begin{pmatrix}
0 & 0 & 1 \\
1 & 0 & 0 \\
0  & 1 & \mu
\end{pmatrix}}
$

 \\
&&&&&\\
\hline
\end{tabular}}
\end{center}

Since the only change of basis matrix producing non-isomorphic four-dimensional evolution algebras is $I_{(3,4)}$, the  matrix $Y$ can be the first one in the table (which is isomorphic to the fourth one in the table), but because of the restriction in the permutations which are allowed, we have to take also the second matrix in the table (which is isomorphic to the last one) and the third one (which is isomorphic to the fifth one).
These are precisely the cases appearing in the second, third and fourth rows in Figure \ref{figuraUno}.

In the following tables, matrices in the same row corresponds to isomorphic evolution algebras. Matrices in different rows are associated to non-isomorphic evolution algebras.

{\begin{figure}[H]
\begin{multicols}{2}
\begin{center}
\scalebox{0.60}{
}
\caption{}\label{figuraUno}
\columnbreak

\scalebox{0.60}{
%
}
\end{center}
\caption{}\label{tableone}
\end{multicols}
\end{figure}}

\begin{figure}[H]
\begin{multicols}{2}
\scalebox{0.600}{
%
}
\caption{}
\columnbreak

\scalebox{0.600}{

%
}
\caption{}
\end{multicols}
\end{figure}

\newpage

\begin{figure}[H]
\begin{multicols}{2}

\scalebox{0.600}{

%
}
\caption{}

\columnbreak

\scalebox{0.600}{

%
}
\caption{}
\end{multicols}
\end{figure}

\newpage

\begin{figure}[H]
\begin{multicols}{2}
\scalebox{0.60}{
%
}
\caption{}\label{tabletwo}
\columnbreak
\scalebox{0.60}{

%
}
\caption{}\label{tablefour}
\end{multicols}
\end{figure}

\subsection{The matrix $U$ has one zero entry.}\label{Case2.1.2}

We are reasoning in a similar way as in Case \ref{Case2.1.1}. Here, the possible matrices are of the form:
\begin{center}
\scalebox{0.60}{$
%
$.}
\end{center}
All of them produce isomorphic evolution algebras and the change of basis matrices are given by $I_{(2,3)}$ (producing the isomorphism between the first and the second) and by $I_{(2,4)}$ (producing the isomorphism between the first and the third). Moreover, by multiplying conveniently one of the vectors in the basis by an scalar, we may assume that the matrix is:
\begin{equation}\label{churritoSiete}
\scalebox{0.60}{$%
$.}
\end{equation}
In order to obtain the tables below, we follow the same procedure as in Case \ref{Case2.1.1}.

{\begin{figure}[H]
		\begin{multicols}{2}
			\begin{center}
				\scalebox{0.60}{
					%
}
					\caption{}
					\columnbreak

					\scalebox{0.60}{
						%
}
					\end{center}
					\caption{}\label{tableone}
				\end{multicols}
			\end{figure}}
			
			\begin{figure}[H]
				\begin{multicols}{2}
					\scalebox{0.600}{
						%
}
							\caption{}
						\end{multicols}
					\end{figure}

					\newpage
					
					\begin{figure}[H]
						\begin{multicols}{2}
							
							\scalebox{0.600}{
								
								%
}
										\caption{}\label{tabletwo}
										\columnbreak
										\scalebox{0.60}{
											
											%
\right)$
												\\
												
												&\\
												\hline
											\end{tabular}}
											\caption{}\label{tablefour}
										\end{multicols}
									\end{figure}

\subsection{The matrix $U$ has no nonzero entries.}\label{Case2.1.3}

Once we have analyzed which are the change of basis matrices that are allowed we get that
$I_\sigma M $ is of the same form as $M$ whenever $\sigma$ belongs to the group $G:= \{{\rm id}, (3,4), (2,3), (2, 3,4), (2,4,3), (2, 4)\}$.

In order to classify we proceed by inserting as matrices $Y$ in \eqref{inicios} all the different matrices associated to simple three-dimensional evolution algebras which appear in the classification in \cite{CCY, CSV2}.

\begin{multicols}{6}
\begin{center}\scalebox{0.60}{$\left(\begin{tabular}{c|ccc}
1 & 1 &  1 & 1 \\
\hline
0 & 0 & 0 & $\omega_{24}$ \\
0 & $\omega_{32}$ & 0 & 0 \\
0 & 0 & $\omega_{43}$ & 0 \\
\end{tabular}\right)$}
\end{center}

\columnbreak

\begin{center}\scalebox{0.60}{$\left(\begin{tabular}{c|ccc}
1 & 1 &  1 & 1 \\
\hline
0 & $\omega_{22}$ & 0 & $\omega_{24}$ \\
0 & $\omega_{32}$ & 0 & 0 \\
0 & 0 & $\omega_{43}$ & 0 \\
\end{tabular}\right)$}
\end{center}

\columnbreak

\begin{center}\scalebox{0.60}{$\left(\begin{tabular}{c|ccc}
1 & 1 &  1 & 1 \\
\hline
0 & 0 & $\omega_{23}$ & $\omega_{24}$ \\
0 & $\omega_{32}$ & 0 & 0 \\
0 & 0 & $\omega_{43}$ & 0 \\
\end{tabular}\right)$}
\end{center}
\columnbreak

\begin{center}\scalebox{0.60}
{$\left(\begin{tabular}{c|ccc}
1 & 1 &  1 & 1 \\
\hline
0 & 0 & $\omega_{23}$ & 0 \\
0 & $\omega_{32}$ & 0 & $\omega_{34}$ \\
0 & $\omega_{42}$ & 0 & $\omega_{44}$ \\
\end{tabular}\right)$}
\end{center}

\columnbreak

\begin{center}\scalebox{0.60}{$\left(\begin{tabular}{c|ccc}
1 & 1 &  1 & 1 \\
\hline
0 & 0 & $\omega_{23}$ & 0 \\
0 & $\omega_{32}$ & 0 & $\omega_{34}$ \\
0 & 0 & $\omega_{43}$ & $\omega_{44}$ \\
\end{tabular}\right)$}
\end{center}

\columnbreak

\begin{center}\scalebox{0.60}{$\left(\begin{tabular}{c|ccc}
1 & 1 &  1 & 1 \\
\hline
0 & $\omega_{22}$ & $\omega_{23}$ & $\omega_{24}$ \\
0 & $\omega_{32}$ & 0 & 0 \\
0 & 0 & $\omega_{43}$ & 0 \\
\end{tabular}\right)$}
\end{center}

\end{multicols}

\begin{multicols}{6}

\begin{center}\scalebox{0.60}{$\left(\begin{tabular}{c|ccc}
1 & 1 &  1 & 1 \\
\hline
0 & $\omega_{22}$ & 0 & $\omega_{24}$ \\
0 & $\omega_{32}$ & $\omega_{33}$ & 0 \\
0 & 0 & $\omega_{43}$ & 0 \\
\end{tabular}\right)$}
\end{center}

\columnbreak

\begin{center}
\scalebox{0.60}
{$\left(\begin{tabular}{c|ccc}
1 & 1 &  1 & 1 \\
\hline
0 & $\omega_{22}$ & 0 & $\omega_{24}$ \\
0 & $\omega_{32}$ & 0 & 0 \\
0 & $\omega_{42}$ & $\omega_{43}$ & 0 \\
\end{tabular}\right)$}
\end{center}

\columnbreak

\begin{center}\scalebox{0.60}{$\left(\begin{tabular}{c|ccc}
1 & 1 &  1 & 1 \\
\hline
0 & 0 & $\omega_{23}$ & $\omega_{24}$ \\
0 & $\omega_{32}$ & 0 & $\omega_{34}$ \\
0 & 0 & $\omega_{43}$ & 0 \\
\end{tabular}\right)$}
\end{center}


\columnbreak

\begin{center}\scalebox{0.60}{$\left(\begin{tabular}{c|ccc}
1 & 1 &  1 & 1 \\
\hline
0 & $\omega_{22}$ & $\omega_{23}$ & 0 \\
0 & 0 & $\omega_{33}$ & $\omega_{34}$ \\
0 & $\omega_{42}$ & 0 & $\omega_{44}$ \\
\end{tabular}\right)$}
\end{center}

\columnbreak

\begin{center}\scalebox{0.60}{$\left(\begin{tabular}{c|ccc}
1 & 1 &  1 & 1 \\
\hline
0 & $\omega_{22}$ & $\omega_{23}$ & $\omega_{24}$ \\
0 & $\omega_{32}$ & 0 & 0 \\
0 & $\omega_{42}$ & 0 & $\omega_{44}$ \\
\end{tabular}\right)$}
\end{center}

\columnbreak

\begin{center}\scalebox{0.60}{$\left(\begin{tabular}{c|ccc}
1 & 1 &  1 & 1 \\
\hline
0 & $\omega_{22}$ & $\omega_{23}$ & $\omega_{24}$ \\
0 & $\omega_{32}$ & 0 & 0 \\
0 & 0 & $\omega_{43}$ & $\omega_{44}$ \\
\end{tabular}\right)$}
\end{center}

\end{multicols}

\begin{multicols}{6}
\begin{center}\scalebox{0.60}{$\left(\begin{tabular}{c|ccc}
1 & 1 &  1 & 1 \\
\hline
0 & 0 & $\omega_{23}$ & $\omega_{24}$ \\
0 & $\omega_{32}$ & $\omega_{33}$ & 0 \\
0 & $\omega_{42}$ & 0 & $\omega_{44}$ \\
\end{tabular}\right)$}
\end{center}

\columnbreak

\begin{center}\scalebox{0.60}{$\left(\begin{tabular}{c|ccc}
1 & 1 &  1 & 1 \\
\hline
0 & 0 & $\omega_{23}$ & $\omega_{24}$ \\
0 & $\omega_{32}$ & $\omega_{33}$ & 0 \\
0 & 0 & $\omega_{43}$ & $\omega_{44}$ \\
\end{tabular}\right)$}
\end{center}

\columnbreak

\begin{center}\scalebox{0.60}{$\left(\begin{tabular}{c|ccc}
1 & 1 &  1 & 1 \\
\hline
0 & 0 & $\omega_{23}$ & $\omega_{24}$ \\
0 & $\omega_{32}$ & 0 & $\omega_{34}$ \\
0 & $\omega_{42}$ & 0 & $\omega_{44}$ \\
\end{tabular}\right)$}
\end{center}

\columnbreak

\begin{center}\scalebox{0.60}{$\left(\begin{tabular}{c|ccc}
1 & 1 &  1 & 1 \\
\hline
0 & 0 & $\omega_{23}$ & $\omega_{24}$ \\
0 & $\omega_{32}$ & 0 & 0 \\
0 & $\omega_{42}$ & $\omega_{43}$ & $\omega_{44}$ \\
\end{tabular}\right)$}
\end{center}

\columnbreak

\begin{center}\scalebox{0.60}{$\left(\begin{tabular}{c|ccc}
1 & 1 &  1 & 1 \\
\hline
0 & $\omega_{22}$ & $\omega_{23}$ & $\omega_{24}$ \\
0 & $\omega_{32}$ & $\omega_{33}$ & 0 \\
0 & 0 & $\omega_{43}$ & 0 \\
\end{tabular}\right)$}
\end{center}

\columnbreak

\begin{center}\scalebox{0.60}{$\left(\begin{tabular}{c|ccc}
1 & 1 &  1 & 1 \\
\hline
0 & $\omega_{22}$ & $\omega_{23}$ & $\omega_{24}$ \\
0 & $\omega_{32}$ & 0 & 0 \\
0 & $\omega_{42}$ & $\omega_{43}$ & 0 \\
\end{tabular}\right)$}
\end{center}

\end{multicols}

\begin{multicols}{6}

\begin{center}\scalebox{0.60}{$\left(\begin{tabular}{c|ccc}
1 & 1 &  1 & 1 \\
\hline
0 & 0 & $\omega_{23}$ & $\omega_{24}$ \\
0 & $\omega_{32}$ & 0 & $\omega_{34}$ \\
0 & $\omega_{42}$ & $\omega_{43}$ & 0 \\
\end{tabular}\right)$}
\end{center}

\columnbreak

\begin{center}\scalebox{0.60}{$\left(\begin{tabular}{c|ccc}
1 & 1 &  1 & 1 \\
\hline
0 & $\omega_{22}$ & $\omega_{23}$ & $\omega_{24}$ \\
0 & $\omega_{32}$ & $\omega_{33}$ & 0 \\
0 & $\omega_{42}$ & 0 & $\omega_{44}$ \\
\end{tabular}\right)$}
\end{center}

\columnbreak

\begin{center}\scalebox{0.60}{$\left(\begin{tabular}{c|ccc}
1 & 1 &  1 & 1 \\
\hline
0 & $\omega_{22}$ & $\omega_{23}$ & $\omega_{24}$ \\
0 & $\omega_{32}$ & $\omega_{33}$ & 0 \\
0 & 0 & $\omega_{43}$ & $\omega_{44}$ \\
\end{tabular}\right)$}
\end{center}
\columnbreak

\begin{center}\scalebox{0.60}{$\left(\begin{tabular}{c|ccc}
1 & 1 &  1 & 1 \\
\hline
0 & $\omega_{22}$ & $\omega_{23}$ & $\omega_{24}$ \\
0 & $\omega_{32}$ & 0 & 0 \\
0 & $\omega_{42}$ & $\omega_{43}$ & $\omega_{44}$ \\
\end{tabular}\right)$}
\end{center}

\columnbreak

\begin{center}\scalebox{0.60}{$\left(\begin{tabular}{c|ccc}
1 & 1 &  1 & 1 \\
\hline
0 & $\omega_{22}$ & $\omega_{23}$ & $\omega_{24}$ \\
0 & $\omega_{32}$ & 0 & $\omega_{34}$ \\
0 & $\omega_{42}$ & 0 & $\omega_{44}$ \\
\end{tabular}\right)$}
\end{center}

\columnbreak

\begin{center}\scalebox{0.60}{$\left(\begin{tabular}{c|ccc}
1 & 1 &  1 & 1 \\
\hline
0 & 0 & $\omega_{23}$ & $\omega_{24}$ \\
0 & $\omega_{32}$ & $\omega_{33}$ & $\omega_{34}$ \\
0 & $\omega_{42}$ & 0 & $\omega_{44}$ \\
\end{tabular}\right)$}
\end{center}

\end{multicols}

\begin{multicols}{6}

\begin{center}\scalebox{0.60}{$\left(\begin{tabular}{c|ccc}
1 & 1 &  1 & 1 \\
\hline
0 & 0 & $\omega_{23}$ & $\omega_{24}$ \\
0 & $\omega_{32}$ & 0 & $\omega_{34}$ \\
0 & $\omega_{42}$ & $\omega_{43}$ & $\omega_{44}$ \\
\end{tabular}\right)$}
\end{center}

\columnbreak


\begin{center}\scalebox{0.60}{$\left(\begin{tabular}{c|ccc}
1 & 1 &  1 & 1 \\
\hline
0 & $\omega_{22}$ & $\omega_{23}$ & $\omega_{24}$ \\
0 & $\omega_{32}$ & $\omega_{33}$ & $\omega_{34}$ \\
0 & $\omega_{42}$ & 0 & $\omega_{44}$ \\
\end{tabular}\right)$}
\end{center}

\columnbreak

\begin{center}\scalebox{0.60}{$\left(\begin{tabular}{c|ccc}
1 & 1 &  1 & 1 \\
\hline
0 & $\omega_{22}$ & $\omega_{23}$ & $\omega_{24}$ \\
0 & $\omega_{32}$ & $\omega_{33}$ & $\omega_{34}$ \\
0 & $\omega_{42}$ & $\omega_{43}$ & 0 \\
\end{tabular}\right)$}
\end{center}

\columnbreak

\begin{center}\scalebox{0.60}{$\left(\begin{tabular}{c|ccc}
1 & 1 &  1 & 1 \\
\hline
0 & $\omega_{22}$ & $\omega_{23}$ & $\omega_{24}$ \\
0 & $\omega_{32}$ & $\omega_{33}$ & $\omega_{34}$ \\
0 & $\omega_{42}$ & $\omega_{43}$ & $\omega_{44}$ \\
\end{tabular}\right)$}
\end{center}

\end{multicols}

\section{The maximal basic ideal is two-dimensional}\label{Case2.2}
%
%
In this case, the structure matrix is as follows:
\begin{equation}\label{churritosConLeche}
\scalebox{0.60}{$\begin{pmatrix}
\omega_{11} & \omega_{12} &  \omega_{13} &  \omega_{14}\\
\omega_{21}  & \omega_{22} &  \omega_{23} &  \omega_{24}\\
0 & 0 &  \omega_{33} &  \omega_{34}\\
0 & 0 &  \omega_{43} &  \omega_{44}\\
\end{pmatrix}$}.
\end{equation}

Notice that $\omega_{13}$, $\omega_{14}$, $\omega_{23}$  and $\omega_{24}$ cannot be zero at the same time because otherwise the algebra will be reducible.

The general form for the matrices we are classifying is
$\begin{pmatrix}
W&  U\\
 0&  Y\\
 \end{pmatrix}
 $,
 where $W=\begin{pmatrix}
\omega_{11} & \omega_{12} \\
\omega_{21}  & \omega_{22} \\
\end{pmatrix}$,
$U=\begin{pmatrix}
\omega_{13} & \omega_{14} \\
\omega_{23}  & \omega_{24} \\
\end{pmatrix}$
 and
 $Y=\begin{pmatrix}
\omega_{33} & \omega_{34} \\
\omega_{43}  & \omega_{44} \\
\end{pmatrix}$.
\medskip

The results that we use in order to classify are the following. By Proposition \ref{OneDimensional}, the number of zeros in $W$, $U$ and $Y$ is invariant. On the other hand, $\vert W \vert , \vert Y \vert \neq 0$ because $M_B\neq 0$ and using \cite[Theorem 2]{Silvester}. Moreover, $A/I$ is basic simple by Proposition \ref{BasicSimpleQuotient}, where $I$ is the maximal 2-basic ideal of $A$.

 First, we fix $W$, which, by the classification of perfect two-dimensional evolution algebras (\cite[Theorem 3.3. (III)]{CCY}), must be one matrix in the following set:
$$ \Gamma:= \left\{

\right\}$$
with $\omega_{ij}\neq 0$ and $\vert W \vert \neq 0$.
\medskip
For each $W\in \Gamma$ we can consider the different $U$ in $\Gamma$. As $Y$ we can consider a matrix in the set $ \Gamma':=\Gamma'_1 \cup \Gamma'_2$, where

$$\Gamma'_1: = \left\{
%
\right\}.$$
Although the evolution algebras having structure matrices
$
%
$
are isomorphic, as parts of the matrix $M_B$ both will appear in some cases, not in all of them. This is the reason to write $\Gamma'$ as that union. 
\bigskip

%
%
\begin{lrbox}{\mybox}
$W= \left( \begin{array}{rr}1 & 0 \\
0 & 1 \end{array} \right)$
\end{lrbox}
\subsection{Take \usebox{\mybox}}\label{Case2.2.1}

\subsubsection{There are three zeros in $U$}\label{Case2.2.1.1} The possibilities for $U$ are:
$%
$, where $\ast$ denotes a nonzero entry. The four cases that appear:

\medskip
\begin{center}
\scalebox{0.60}{
$
%
$}
\end{center}

\noindent
are isomorphic to the first one via the permutations: $(3,4)$, $(1,2)$ and $(1,2)(3,4)$, respectively. The possible evolution algebras are
\medskip

\begin{center}
	\scalebox{0.60}{$
\left(\begin{tabular}{c|c}
$\begin{matrix}
1 & 0 \\ 0 & 1
\end{matrix}$&
$\begin{matrix}
0 & 0 \\ 0 & \omega_{24}
\end{matrix}$\\
\hline
$\begin{matrix}
0& 0 \\ 0 & 0
\end{matrix}$& $Y$
\end{tabular}\right)$}
\end{center}

\noindent
where $Y \in \Gamma'$. We have four families of mutually non-isomorphic evolution algebras.
\medskip
%
%

\subsubsection{There are two zeros in $U$}\label{Case2.2.1.2}  In this case $U$ can be:
$$
%
,
$$
where $\ast$ denotes a nonzero entry. The six cases appearing are:

\medskip

\begin{center}
\scalebox{0.60}{
$
%
.
$}

\end{center}

\medskip

The first and the second ones are isomorphic, as well as the third and the fourth ones, and the fifth and the last ones. The permutations giving these isomorphisms are $(1, 2)$ in the first case and $(3, 4)$ in the other cases. In order to classify we choose the matrices in the  odd position. The possible evolution algebras are

\begin{center}
\scalebox{0.60}{$
\left(\begin{tabular}{c|c}
$\begin{matrix}
1 & 0 \\ 0 & 1
\end{matrix}$&
$\begin{matrix}
0 & 0 \\ \omega_{23} & \omega_{24}
\end{matrix}$\\
\hline
$\begin{matrix}
0& 0 \\ 0 & 0
\end{matrix}$& $Y_1$
\end{tabular}\right),
\left(\begin{tabular}{c|c}
$\begin{matrix}
1 & 0 \\ 0 & 1
\end{matrix}$&
$\begin{matrix}
0 & \omega_{14} \\ 0 & \omega_{24}
\end{matrix}$\\
\hline
$\begin{matrix}
0& 0 \\ 0 & 0
\end{matrix}$& $Y_2$
\end{tabular}\right),
\left(\begin{tabular}{c|c}
$\begin{matrix}
1 & 0 \\ 0 & 1
\end{matrix}$&
$\begin{matrix}
0 & \omega_{14}\\ \omega_{23} & 0
\end{matrix}$\\
\hline
$\begin{matrix}
0& 0 \\ 0 & 0
\end{matrix}$& $Y_3$
\end{tabular}\right),$}
\end{center}

\noindent
where $Y_1, Y_3 \in \Gamma_1'$ and $Y_2\in \Gamma'$. We have ten families of mutually non-isomorphic evolution algebras.

\medskip

\subsubsection{There is one zero in $U$.}\label{Case2.2.1.3}
The possibilities for $U$ are:
$%
$, where $\ast$ denotes a nonzero entry. The four cases that appear
\medskip

\begin{equation}\label{CuatroMatrices}
\scalebox{0.60}{
$
%
$}
\end{equation}
are isomorphic to the first one via the permutations $(3,4)$, $(1,2)$ and $(1,2)(3,4)$, respectively. The possible evolution algebras are

\begin{center}
\scalebox{0.60}{$
\left(\begin{tabular}{c|c}
$\begin{matrix}
1 & 0 \\ 0 & 1
\end{matrix}$&
$\begin{matrix}
\omega_{13} & \omega_{14} \\ \omega_{23} & 0
\end{matrix}$\\
\hline
$\begin{matrix}
0& 0 \\ 0 & 0
\end{matrix}$& $Y$
\end{tabular}\right)$,}
\end{center}

\noindent
where $Y \in \Gamma'$. We have four families of mutually non-isomorphic evolution algebras.
\medskip

\subsubsection{There are no zeros in $U$}\label{Case2.2.1.4} There is only one possibility for $U$ and, therefore, the matrix in \eqref{churritosConLeche} takes the form:

\begin{center}
	\scalebox{0.60}{$\left(\begin{tabular}{c|c}
$\begin{matrix}
1 & 0 \\ 0 & 1
\end{matrix}$&
$\begin{matrix}
\omega_{13} & \omega_{14} \\ \omega_{23} &  \omega_{24}
\end{matrix}$\\
\hline
$\begin{matrix}
0& 0 \\ 0 & 0
\end{matrix}$& $Y$
\end{tabular}\right)$,}

\end{center}

\noindent
where $Y \in \Gamma_1'$. We have three families of mutually non-isomorphic evolution algebras.
\medskip

\bigskip
%
%

\begin{lrbox}{\mybox}
$W= \left( \begin{array}{cc}1 & 0 \\
\omega_{21} & 1 \end{array} \right)$
\end{lrbox}
\subsection{Take \usebox{\mybox}}\label{Case2.2.2}

\subsubsection{There are three zeros in $U$}\label{Case 2.2.2.1} As in Case \ref{Case2.2.1.1}, there are four possibilities for $U$ and the structure matrix has the form

\medskip

\begin{center}
\scalebox{0.60}{
$
%
.
$}
\end{center}
\medskip

The first matrix is isomorphic to the second one and the third one is isomorphic to the last one, in both cases, via the permutation $(3,4)$.
For the classification we select  the first  and the third matrices. The possible evolution algebras are
\medskip

\begin{center}
\scalebox{0.60}{
$
\left(\begin{tabular}{c|c}
$\begin{matrix}
1 & 0 \\ \omega_{21} & 1
\end{matrix}$&
$\begin{matrix}
0 & 0 \\ 0 & \omega_{24}
\end{matrix}$\\
\hline
$\begin{matrix}
0& 0 \\ 0 & 0
\end{matrix}$& $Y$
\end{tabular}\right),
\quad
\left(\begin{tabular}{c|c}
$\begin{matrix}
1 & 0 \\ \omega_{21} & 1
\end{matrix}$&
$\begin{matrix}
0 & \omega_{14} \\ 0 & 0
\end{matrix}$\\
\hline
$\begin{matrix}
0& 0 \\ 0 & 0
\end{matrix}$& $Y$
\end{tabular}\right).
$}
\end{center}
\medskip

For any election of $Y$ in $\Gamma'$ we have mutually non-isomorphic evolution algebras. The number of families is eight.
\medskip

\subsubsection{There are two zeros in $U$}\label{Case2.2.2.2} As in Case \ref{Case2.2.1.2} there are six possibilities for $U$. The matrix in \eqref{churritosConLeche} becomes:
\medskip

\begin{equation*}\label{SeisMatrices}
\scalebox{0.60}{
$
%
.
$}
\end{equation*}
\medskip

The third and the fourth are isomorphic, as well as the fifth and the last one. The permutation giving these isomorphisms is  $(3, 4)$.

We fix the first, the second, the third and the fifth matrices. The possible evolution algebras are
\medskip

\begin{center}
\scalebox{0.60}{
$
\left(\begin{tabular}{c|c}
$\begin{matrix}
1 & 0 \\ \omega_{21} & 1
\end{matrix}$&
$\begin{matrix}
0 & 0 \\ \omega_{23} & \omega_{24}
\end{matrix}$\\
\hline
$\begin{matrix}
0& 0 \\ 0 & 0
\end{matrix}$& $Y_1$
\end{tabular}\right),
\left(\begin{tabular}{c|c}
$\begin{matrix}
1 & 0 \\ \omega_{21}  & 1
\end{matrix}$&
$\begin{matrix}
 \omega_{13}  & \omega_{14} \\ 0 & 0
\end{matrix}$\\
\hline
$\begin{matrix}
0& 0 \\ 0 & 0
\end{matrix}$& $Y_2$
\end{tabular}\right),
\left(\begin{tabular}{c|c}
$\begin{matrix}
1 & 0 \\ \omega_{21}  & 1
\end{matrix}$&
$\begin{matrix}
0 & \omega_{14}\\ 0&  \omega_{24}
\end{matrix}$\\
\hline
$\begin{matrix}
0& 0 \\ 0 & 0
\end{matrix}$& $Y_3$
\end{tabular}\right),
\left(\begin{tabular}{c|c}
$\begin{matrix}
1 & 0 \\ \omega_{21}  & 1
\end{matrix}$&
$\begin{matrix}
0 & \omega_{14}\\ \omega_{23} & 0
\end{matrix}$\\
\hline
$\begin{matrix}
0& 0 \\ 0 & 0
\end{matrix}$& $Y_4$
\end{tabular}\right)$,}
\end{center}
\medskip

\noindent
where $Y_1, Y_2 \in \Gamma_1'$ and $Y_3, Y_4\in \Gamma'$. We have fourteen families of mutually non-isomorphic evolution algebras.
\medskip

\subsubsection{There is one zero in $U$}\label{Case2.2.2.3}
As in Case \ref{Case2.2.1.3}, there are four possibilities for $U$ and the structure matrix can have the form
\medskip

\begin{center}
\scalebox{0.60}{
$
%
.
$}
\end{center}

\medskip

The first one is isomorphic to the second one and the third one is isomorphic to the last one, via the permutation $(3,4)$. For the classification we fix the first matrix and the third one.
The possible evolution algebras are

\begin{center}
\scalebox{0.60}{$
\left(\begin{tabular}{c|c}
$\begin{matrix}
1 & 0 \\ \omega_{21} & 1
\end{matrix}$&
$\begin{matrix}
 \omega_{13} &  \omega_{14} \\  \omega_{23} &0
\end{matrix}$\\
\hline
$\begin{matrix}
0& 0 \\ 0 & 0
\end{matrix}$& $Y$
\end{tabular}\right),
\quad
\left(\begin{tabular}{c|c}
$\begin{matrix}
1 & 0 \\ \omega_{21} & 1
\end{matrix}$&
$\begin{matrix}
 \omega_{13} & 0 \\  \omega_{23} &  \omega_{24}
\end{matrix}$\\
\hline
$\begin{matrix}
0& 0 \\ 0 & 0
\end{matrix}$& $Y$
\end{tabular}\right).
$}
\end{center}
\medskip

For any election of $Y$ in $\Gamma'$ we have mutually non-isomorphic evolution algebras. The number of families is eight.

\medskip

 \subsubsection{There are no zeros in $U$}\label{Case2.2.2.4}
There is only one possibility for $U$ and, therefore, the matrix in \eqref{churritosConLeche} takes the form:
\medskip

\begin{center}
\scalebox{0.60}{
$\left(\begin{tabular}{c|c}
$\begin{matrix}
1 & 0 \\ \omega_{21} & 1
\end{matrix}$&
$\begin{matrix}
\omega_{13} & \omega_{14} \\ \omega_{23} &  \omega_{24}
\end{matrix}$\\
\hline
$\begin{matrix}
0& 0 \\ 0 & 0
\end{matrix}$& $Y$
\end{tabular}\right),$}
\end{center}
\medskip

\noindent
where $Y \in \Gamma_1'$. We have three families of mutually non-isomorphic evolution algebras.
\medskip

\bigskip
%
%
\begin{lrbox}{\mybox}
$W= \left( \begin{array}{cc}1 & \omega_{12} \\
\omega_{21} & 1 \end{array} \right)$
\end{lrbox}
\subsection{Take \usebox{\mybox}}\label{Case2.2.3}

\subsubsection{There are three zeros in $U$}\label{Case2.2.3.1} Again, we have four possibilities for $U$. The four cases that appear
\medskip

\begin{center}
\scalebox{0.60}{
$
%
$,}
\end{center}

\noindent
are isomorphic to the first one via the permutations $(3,4)$, $(1,2)$ and $(1,2)(3,4)$, respectively.  For the classification we fix the first matrix. The possible evolution algebras are
\medskip

\begin{center}
\scalebox{0.60}{$
\left(\begin{tabular}{c|c}
$\begin{matrix}
1 & \omega_{12} \\ \omega_{21} & 1
\end{matrix}$&
$\begin{matrix}
0 & 0 \\ 0 & \omega_{24}
\end{matrix}$\\
\hline
$\begin{matrix}
0& 0 \\ 0 & 0
\end{matrix}$& $Y$
\end{tabular}\right)$,}
\end{center}
\medskip

\noindent
where $Y \in \Gamma'$. We have four families of mutually non-isomorphic evolution algebras.
\medskip

\subsubsection{There are two zeros in $U$}\label{Case2.2.3.2} Again there are six possibilities for $U$. The matrix in \eqref{churritosConLeche} can be:

\begin{equation*}\label{SeisMatrices}
\scalebox{0.60}{
$
%
.
$}
\end{equation*}
The first one and the second one are isomorphic, as well as the third and the fourth and the fifth and the last one. The permutations giving these isomorphisms are   $(1,2)$ in the first case and $(3, 4)$ in the other cases.

We fix the first, the third and the fifth matrices. The possible evolution algebras are
\medskip

\begin{center}
\scalebox{0.60}{$
\left(\begin{tabular}{c|c}
$\begin{matrix}
1 & \omega_{12}  \\ \omega_{21} & 1
\end{matrix}$&
$\begin{matrix}
0 & 0 \\ \omega_{23} & \omega_{24}
\end{matrix}$\\
\hline
$\begin{matrix}
0& 0 \\ 0 & 0
\end{matrix}$& $Y_1$
\end{tabular}\right),
\left(\begin{tabular}{c|c}
$\begin{matrix}
1 & \omega_{12} \\ \omega_{21}  & 1
\end{matrix}$&
$\begin{matrix}
0 & \omega_{14}\\ 0&  \omega_{24}
\end{matrix}$\\
\hline
$\begin{matrix}
0& 0 \\ 0 & 0
\end{matrix}$& $Y_2$
\end{tabular}\right),
\left(\begin{tabular}{c|c}
$\begin{matrix}
1 & \omega_{12} \\ \omega_{21}  & 1
\end{matrix}$&
$\begin{matrix}
0 & \omega_{14}\\ \omega_{23} & 0
\end{matrix}$\\
\hline
$\begin{matrix}
0& 0 \\ 0 & 0
\end{matrix}$& $Y_3$
\end{tabular}\right),$}
\end{center}
\medskip

\noindent
where $Y_1, Y_3 \in \Gamma_1'$ and $Y_2\in \Gamma'$. We have ten families of mutually non-isomorphic evolution algebras.
\medskip

\subsubsection{There is one zero in $U$}\label{Case2.2.3.3}
Again, we have four possibilities for $U$. The four cases that appear
\medskip

\begin{center}
\scalebox{0.60}{
$
%
,
$}
\end{center}

\noindent
are isomorphic to the first one via the permutations $(3,4)$, $(1,2)$ and $(1,2)(3,4)$, respectively.  For the classification we fix the first matrix. The possible evolution algebras are
\medskip

\begin{center}
\scalebox{0.60}{$
\left(\begin{tabular}{c|c}
$\begin{matrix}
1 & \omega_{12} \\ \omega_{21} & 1
\end{matrix}$&
$\begin{matrix}
\omega_{13}& \omega_{14} \\ \omega_{23} &0
\end{matrix}$\\
\hline
$\begin{matrix}
0& 0 \\ 0 & 0
\end{matrix}$& $Y$
\end{tabular}\right)$,}
\end{center}
\medskip

\noindent
where $Y \in \Gamma'$. We have four families of mutually non-isomorphic evolution algebras.

\medskip

\subsubsection{There are no zeros in $U$}\label{Case2.2.3.4} There is only one possibility for $U$ and, therefore, the matrix in \eqref{churritosConLeche} takes the form:
\medskip

\begin{center}
\scalebox{0.60}{$\left(\begin{tabular}{c|c}
$\begin{matrix}
1 & \omega_{12} \\ \omega_{21} & 1
\end{matrix}$&
$\begin{matrix}
\omega_{13} & \omega_{14} \\ \omega_{23} &  \omega_{24}
\end{matrix}$\\
\hline
$\begin{matrix}
0& 0 \\ 0 & 0
\end{matrix}$& $Y$
\end{tabular}\right)$,}
\end{center}
\medskip

\noindent
where $Y \in \Gamma_1'$. We have three families of mutually non-isomorphic evolution algebras.
\medskip

\bigskip
%
%
\begin{lrbox}{\mybox}
$W= \left( \begin{array}{cc}0 & 1 \\
1 & 0 \end{array} \right)$
\end{lrbox}
\subsection{Take \usebox{\mybox}}\label{Case2.2.4}

\subsubsection{There are three zeros in $U$}\label{Case2.2.4.1}
 Again we have four possibilities for $U$:
\medskip

\begin{center}
\scalebox{0.60}{
$
%
.
$}
\end{center}

\medskip

All of them are isomorphic to the first one via the permutations $(3,4)$, $(1,2)$ and $(1,2)(3,4)$, respectively.  For the classification we fix the first matrix. The possible evolution algebras are

\medskip

\begin{center}
	\scalebox{0.60}{$
\left(\begin{tabular}{c|c}
$\begin{matrix}
0 & 1 \\ 1 & 0
\end{matrix}$&
$\begin{matrix}
0 & 0 \\ 0 & \omega_{24}
\end{matrix}$\\
\hline
$\begin{matrix}
0& 0 \\ 0 & 0
\end{matrix}$& $Y$
\end{tabular}\right)$,}
\end{center}

\noindent
where $Y \in \Gamma'$. We have four families of mutually non-isomorphic evolution algebras.

\medskip

 \subsubsection{There are two zeros in $U$}\label{Case2.2.4.2}
 Again, there are six possibilities for $U$.

The six cases appearing are:

\medskip

\begin{center}
\scalebox{0.605}{
$
%
.
$}
\end{center}

\medskip
The first and the second are isomorphic, as well as the third and the fourth one, and the fifth and the last one. The permutations giving these isomorphisms are $(1, 2)$ in first case and $(3, 4)$ in the other cases. For the classification we choose the matrices appearing in the odd positions. The possible evolution algebras are

\begin{center}
	\scalebox{0.60}{$
\left(\begin{tabular}{c|c}
$\begin{matrix}
0 & 1 \\ 1 & 0
\end{matrix}$&
$\begin{matrix}
0 & 0 \\ \omega_{23} & \omega_{24}
\end{matrix}$\\
\hline
$\begin{matrix}
0& 0 \\ 0 & 0
\end{matrix}$& $Y_1$
\end{tabular}\right),
\left(\begin{tabular}{c|c}
$\begin{matrix}
0 & 1 \\ 1 & 0
\end{matrix}$&
$\begin{matrix}
0 & \omega_{14} \\ 0 & \omega_{24}
\end{matrix}$\\
\hline
$\begin{matrix}
0& 0 \\ 0 & 0
\end{matrix}$& $Y_2$
\end{tabular}\right),
\left(\begin{tabular}{c|c}
$\begin{matrix}
0 & 1 \\ 1 & 0
\end{matrix}$&
$\begin{matrix}
0 & \omega_{14}\\ \omega_{23} & 0
\end{matrix}$\\
\hline
$\begin{matrix}
0& 0 \\ 0 & 0
\end{matrix}$& $Y_3$
\end{tabular}\right),$}
\end{center}

\noindent
where $Y_1, Y_3 \in \Gamma_1'$ and $Y_2\in \Gamma'$. We have ten families of mutually non-isomorphic evolution algebras.
\medskip

\subsubsection{There is one zero in $U$}\label{Case2.2.4.3}
 Again we have four possibilities for $U$:
\medskip

\begin{center}
\scalebox{0.60}{
$
%
.
$}
\end{center}

\medskip

All of them are isomorphic to the first one via the permutations $(3,4)$, $(1,2)$ and $(1,2)(3,4)$, respectively.  For the classification we fix the first matrix. The possible evolution algebras are

\begin{center}
	\scalebox{0.60}{$
\left(\begin{tabular}{c|c}
$\begin{matrix}
0 & 1 \\ 1 & 0
\end{matrix}$&
$\begin{matrix}
 \omega_{13} &  \omega_{14} \\  \omega_{23} & 0
\end{matrix}$\\
\hline
$\begin{matrix}
0& 0 \\ 0 & 0
\end{matrix}$& $Y$
\end{tabular}\right)$,}
\end{center}

\noindent
where $Y \in \Gamma'$. We have four families of mutually non-isomorphic evolution algebras.

\medskip

\subsubsection{There are no zeros in $U$}\label{Case2.2.4.4}
 There is only one possibility for $U$ and, therefore, the matrix in \eqref{churritosConLeche} takes the form:

\begin{center}
\scalebox{0.60}{$\left(\begin{tabular}{c|c}
$\begin{matrix}
0 & 1 \\ 1 & 0
\end{matrix}$&
$\begin{matrix}
\omega_{13} & \omega_{14} \\ \omega_{23} &  \omega_{24}
\end{matrix}$\\
\hline
$\begin{matrix}
0& 0 \\ 0 & 0
\end{matrix}$& $Y$
\end{tabular}\right)$,}
\end{center}

\noindent
where $Y \in \Gamma_1'$. We have three families of mutually non-isomorphic evolution algebras.
\medskip

\bigskip
%
%
\begin{lrbox}{\mybox}
$W= \left( \begin{array}{cc}0 & 1 \\
1 & \omega_{22} \end{array} \right)$
\end{lrbox}
\subsection{Take \usebox{\mybox}}\label{Case2.2.5}

 \subsubsection{There are three zeros in $U$}\label{Case2.2.5.1}
In this case there are four possibilities for $U$ and the matrix in \eqref{churritosConLeche}  can take the form:
\medskip

\begin{center}
\scalebox{0.60}{
$
%
.
$}
\end{center}

\medskip

The first one is isomorphic to the second one and the third one is isomorphic to the last one, via, in both cases, the permutation $(3,4)$.
We fix the first matrix and the third one. The possible evolution algebras are:
\medskip

\begin{center}
	\scalebox{0.60}{$
\left(\begin{tabular}{c|c}
$\begin{matrix}
0 & 1 \\ 1 & \omega_{22}
\end{matrix}$&
$\begin{matrix}
0 & 0 \\ 0 & \omega_{24}
\end{matrix}$\\
\hline
$\begin{matrix}
0& 0 \\ 0 & 0
\end{matrix}$& $Y$
\end{tabular}\right)$},
 \scalebox{0.60}{$\left(\begin{tabular}{c|c}
$\begin{matrix}
0 & 1 \\ 1 &  \omega_{22}
\end{matrix}$&
$\begin{matrix}
0 & \omega_{14} \\ 0 & 0
\end{matrix}$\\
\hline
$\begin{matrix}
0& 0 \\ 0 & 0
\end{matrix}$& $Y$
\end{tabular}\right)$}.

\end{center}

\medskip

For any election of $Y$ in $\Gamma'$ we have mutually non-isomorphic evolution algebras. The number of families is eight.
\medskip

\subsubsection{There are two zeros in $U$}\label{Case2.2.5.2}
As before, $U$ has six possibilities.

The matrix in \eqref{churritosConLeche} becomes:
\begin{equation*}
\scalebox{0.60}{
$
%
.
$}
\end{equation*}
The third and the fourth are isomorphic, as well as the fifth and the last one. The permutation giving these isomorphisms is  $(3, 4)$.

We fix the first, the second, the third and the fifth matrices. The possible evolution algebras are

\begin{center}
	\scalebox{0.60}{
$
\left(\begin{tabular}{c|c}
$\begin{matrix}
0 & 1 \\ 1 & \omega_{22}
\end{matrix}$&
$\begin{matrix}
0 & 0 \\ \omega_{23} & \omega_{24}
\end{matrix}$\\
\hline
$\begin{matrix}
0& 0 \\ 0 & 0
\end{matrix}$& $Y_1$
\end{tabular}\right),
\left(\begin{tabular}{c|c}
$\begin{matrix}
0 & 1 \\ 1 & \omega_{22}
\end{matrix}$&
$\begin{matrix}
 \omega_{13}  & \omega_{14} \\ 0 & 0
\end{matrix}$\\
\hline
$\begin{matrix}
0& 0 \\ 0 & 0
\end{matrix}$& $Y_2$
\end{tabular}\right),
\left(\begin{tabular}{c|c}
$\begin{matrix}
0 & 1\\ 1 & \omega_{22}
\end{matrix}$&
$\begin{matrix}
0 & \omega_{14}\\ 0&  \omega_{24}
\end{matrix}$\\
\hline
$\begin{matrix}
0& 0 \\ 0 & 0
\end{matrix}$& $Y_3$
\end{tabular}\right),
\left(\begin{tabular}{c|c}
$\begin{matrix}
0& 1 \\1  & \omega_{22}
\end{matrix}$&
$\begin{matrix}
0 & \omega_{14}\\ \omega_{23} & 0
\end{matrix}$\\
\hline
$\begin{matrix}
0& 0 \\ 0 & 0
\end{matrix}$& $Y_4$
\end{tabular}\right)$},
\end{center}

\noindent
where $Y_1, Y_2 \in \Gamma_1'$ and $Y_3, Y_4\in \Gamma'$. We have fourteen families of mutually non-isomorphic evolution algebras.
\medskip

\subsubsection{There is one zero in $U$}\label{Case2.2.5.3}

There are four possibilities for $U$.
In this case the matrix in \eqref{churritosConLeche}  can take the form:

\begin{center}
\scalebox{0.60}{
$
%
.
$}
\end{center}

\medskip

The first one is isomorphic to the second one and the third one is isomorphic to the last one, via, in both cases, the permutation $(3,4)$.
For the classification we fix the first one and the third one:

\begin{center}
\scalebox{0.60}{$
\left(\begin{tabular}{c|c}
$\begin{matrix}
0 & 1 \\ 1 & \omega_{22}
\end{matrix}$&
$\begin{matrix}
 \omega_{13} &  \omega_{14} \\  \omega_{23} &0
\end{matrix}$\\
\hline
$\begin{matrix}
0& 0 \\ 0 & 0
\end{matrix}$& $Y$
\end{tabular}\right)$},
\scalebox{0.60}{$
\left(\begin{tabular}{c|c}
$\begin{matrix}
0 & 1 \\ 1 &  \omega_{22}
\end{matrix}$&
$\begin{matrix}
 \omega_{13} & 0 \\  \omega_{23} &  \omega_{24}
\end{matrix}$\\
\hline
$\begin{matrix}
0& 0 \\ 0 & 0
\end{matrix}$& $Y$
\end{tabular}\right)$}.
\end{center}

For any election of $Y$ in $\Gamma'$ we have mutually non-isomorphic evolution algebras. The number of families is eight.
\medskip

\subsubsection{There are no zeros in $U$}\label{Case2.2.5.4} There is only one possibility for $U$ and, therefore, the matrix in \eqref{churritosConLeche} takes the form:

\begin{center}
	\scalebox{0.60}{$\left(\begin{tabular}{c|c}
$\begin{matrix}
0 & 1 \\ 1 & \omega_{22}
\end{matrix}$&
$\begin{matrix}
\omega_{13} & \omega_{14} \\ \omega_{23} &  \omega_{24}
\end{matrix}$\\
\hline
$\begin{matrix}
0& 0 \\ 0 & 0
\end{matrix}$& $Y$
\end{tabular}\right)$,}
\end{center}

\noindent
where $Y \in \Gamma_1'$. We have three families of mutually non-isomorphic evolution algebras.
\medskip

\section{The evolution algebra has a  maximal 3-basic ideal}\label{SectionThreeDimensional}

Assume that $I$ is a maximal 3-basic ideal of $A$. Write:
$B=\{e_1, e_2, e_3, e_4\}$, $I={\rm lin}\{e_1, e_2, e_3\}$ and
\begin{equation}\label{horchata}M_B =\scalebox{0.60}{$%
$}.
\end{equation}
Note that, in this case, the matrices $U$ and $W$ appearing in \eqref{inicios} are:
\begin{equation*}\label{UeY}
U= \scalebox{0.60}{$%
$}.
\end{equation*}

As Example \ref{EjemplosDeTodo} (i) shows, the number of zeros of the matrices $U$ and $W$ is not necessarily invariant. This happens only when the conditions in Proposition \ref{NoInvariantes} are satisfied.

Therefore,  in order to classify the evolution algebras which appear, we need a different property. To find it,  we pay attention to the change of basis matrices (which are permutations of the basis up to scalars, as explained before) which are allowed. Looking at the matrix $M_B$ it seems natural to think that the  possible permutations are those which fix $e_4$. This is the case, as we show in Proposition \ref{ObservCuatro}

\begin{proposition}\label{NoInvariantes}
Let $A$ be a perfect $n$-dimensional evolution algebra having an $(n-1)$-basic ideal $I$ and let  $B=\{e_1, \dots, e_n\}$ be a natural basis. Assume (there is no loss in generality in doing this) that $I=\langle \{e_1, \dots, e_{n-1}\}\rangle$. Let

$$
M_B=\left(\begin{tabular}{c|c}
$W$
&
$U$
\\
\hline
$\begin{matrix}
0& \cdots & 0
\end{matrix}$
&
 $\omega_{nn}$
\end{tabular}\right),$$
where $W \in M_3(\K)$, $U\in M_{3\times1}(\K)$ and $\omega_{nn}\in \K^\times$.
Then, there exists a natural basis $B'$ of $A$ with
$$
M_{B'}=\left(\begin{tabular}{c|c}
$W'$
&
$U'$
\\
\hline
$\begin{matrix}
0& \cdots & 0
\end{matrix}$
&
 $\omega'_{nn}$
\end{tabular}\right),$$
where $W' \in M_3(\K)$, $U'\in M_{3\times1}(\K)$ and  $\omega'_{nn}\in \K^\times$,
such that the number of zeros in $W$ and in $W'$ does not coincide  (and, consequently the number of zeros in $U$ and in $U'$ is different) if and only if there exists $i\in \{1, \dots, n-1\}$ such that $i\notin D^1(j)$ for any $j\in \{1, \dots, n\}\setminus \{i\}$ and $\vert D^1(i)\vert \neq \vert D^1(n)\vert$.
\end{proposition}
\begin{proof}
Assume first that the number of zeros is not invariant for some basis $B'$ such that
$$
M_{B'}=\left(\begin{tabular}{c|c}
$W'$
&
$U'$
\\
\hline
$\begin{matrix}
0& \cdots & 0
\end{matrix}$
&
 $\omega'_{nn}$
\end{tabular}\right).$$
 Since (up to scalars) the only changes of basis are the permutation matrices (by \cite[Theorem 4.4]{EL}), there exists $i\in  \{1, \dots, n-1\}$ and $\sigma \in S_n$ such that $\sigma(i)=n$,
 and $B'=\{e_{\sigma(i)} \ \vert \ i\in \{1, \dots, n\}\}$. Taking into account the last row in $M_{B'}$ we obtain $i\notin D^1(j)$ for any $j\in \{1, \dots, n\}\setminus \{i\}$. Since we are assuming that the number of zeros in $W$ and in $W'$ is different, necessarily $\vert D^1(i)\vert \neq \vert D^1(n)\vert$.

 Now we prove the reverse implication.
Assume that there exists $i$ satisfying the conditions in the statement. Take $\sigma\in S_n$ such that $\sigma(i)=n$ and $\sigma(n)=i$ and define $B'=\{e_{\sigma(i)} \ \vert \ i\in \{1, \dots, n\}\}$. Then
$$
M_{B'}=\left(\begin{tabular}{c|c}
$W'$
&
$U'$
\\
\hline
$\begin{matrix}
0& \cdots & 0
\end{matrix}$
&
 $\omega'_{nn}$
\end{tabular}\right),$$
where $\omega'_{nn}\in \K^\times$ and the number of zeros in $W'$ is not the same as the number of zeros in $W$.
\end{proof}

\begin{remark}\label{Cambio4}
\rm
 In the proof of Proposition \ref{NoInvariantes} and considering $n=4$, it is shown that  the condition $i\notin D^1(j)$ for any $j\in \{1, \dots, 4\}\setminus \{i\}$ implies that the change of basis obtained from $B$ by
puting $e_i$ instead of $e_4$ produces an structure matrix having the same form as \eqref{horchata}. This is the reason to classify taking into account this condition.
\end{remark}

Take $n=4$ in Proposition \ref{NoInvariantes}, and use Remark \ref{Cambio4} to get the following.

\begin{proposition}\label{ObservCuatro}
Let $A$ be a perfect $4$-dimensional evolution algebra having a $3$-basic ideal $I$ and let  $B=\{e_1, \dots, e_4\}$ be a natural basis. Assume (there is no loss in generality in doing this) that $I=\langle \{e_1, \dots, e_{3}\}\rangle$. Let

$$
M_B=\left(\begin{tabular}{c|c}
$W$
&
$U$
\\
\hline
$\begin{matrix}
0& \cdots & 0
\end{matrix}$
&
 $\omega_{44}$
\end{tabular}\right),$$
where $W \in M_3(\K)$, $U\in M_{3\times1}(\K)$ and $\omega_{44}\in \K^\times$.

 Then, there exists $i\in \{1, \dots, 3\}$ such that $i\notin D^1(j)$ for any $j\in \{1, \dots, 4\}\setminus \{i\}$  if and only if there exist  $k, l \in \{1, 2, 3\} \setminus \{ i\}$, with $k\neq l$, such that ${\rm lin}\{e_k, e_l\}$ is a two-dimensional evolution ideal of $A$ and ${\rm lin}\{e_k, e_l, e_4\}$ is a three dimensional evolution ideal.
 \end{proposition}

When  the conditions in Proposition \ref{ObservCuatro} are satisfied, then the evolution algebra $A$, which contains a 3-basic ideal $I$, also contains another 3-basic ideal, say $J$, such that $I\cap J$ is a 2-basic ideal of both, $I$ and $J$. We will use this fact for the classification.

\begin{remark}
\rm
Note that for perfect evolution algebras, the intersection of basic ideals is again a basic ideal.
\end{remark}

In this section we are assuming that $A$  has a 3-basic ideal $I$. We will distinguish whether or not $I$ has a 2-basic ideal. In the affirmative, we will distinguish whether or not this 2-basic ideal is contained in another 3-basic ideal.

\begin{definition}
\rm
An evolution algebra $A$ will be said to satisfy \emph{Condition} (3,2,3) if $A$ has two different 3-basic ideals $I$ and $J$ such that $I\cap J$ is a 2-basic ideal of $I$ and also of $J$.
\end{definition}

\subsection{$A$ has a  3-basic ideal which has a 2-basic ideal.}
We may assume that $\{e_1, e_2\}$ is a 2-basic ideal of $I$, where $I$ is a 3-basic ideal having basis $\{e_1, e_2, e_3\}$.

In the tables that follow we will analyze which types of $U$ and $W$ produce a structure matrix such that $A$ is irreducible, where $W$ is the matrix of $I$. These matrices appear in the classification of three-dimensional evolution algebras (see \cite{CCY, CSV2}). Whenever we obtain an irreducible evolution algebra $A$, we  add $\ast$ when  $A$ does not satisfy Condition (3,2,3). We remark again that this means that $e_4$ cannot be changed to any other element in the natural basis in order to produce another matrix having the same form.
\smallskip

\begin{figure}[H]
\begin{center}
\scalebox{0.6}{
\begin{tabular}{|c||c||c||c||c||c||c||c|}
\hline
&&&&&&&\\
$W$ & $U= \tiny{\begin{pmatrix}1 \\0 \\0\end{pmatrix}}$ & $U= \tiny{\begin{pmatrix}0 \\1 \\0\end{pmatrix}}$& $U=\tiny{\begin{pmatrix}0\\0 \\1\end{pmatrix}}$& $U= \tiny{\begin{pmatrix}1 \\1 \\0\end{pmatrix}}$& $U=\tiny{\begin{pmatrix}1 \\0 \\1\end{pmatrix}}$& $U=\tiny{\begin{pmatrix}0 \\1 \\1\end{pmatrix}}$ & $U=\tiny{\begin{pmatrix}1 \\1 \\1\end{pmatrix}}$ \\
&&&&&&&\\
\hline
\hline
&&&&&&&\\[-0.2cm]
$\begin{pmatrix}
\omega_{11} & 0 &  0 \\
0 & \omega_{22} & 0 \\
0 & 0 & \omega_{33} \\
\end{pmatrix}$
 &
 Reducible
 &
 Reducible
 &
 Reducible
 &
 Reducible
 &
 Reducible
 &
 Reducible
 &
 Irreducible*
 \\
&&&&&&&\\[-0.2cm]
\hline
\hline
&&&&&&&\\[-0.2cm]
$\begin{pmatrix}
0 & \omega_{12} &  0 \\
\omega_{21} & 0 & 0 \\
0 & 0 & \omega_{33} \\
\end{pmatrix}$
 &
 Reducible
 &
 Reducible
 &
 Reducible
 &
 Reducible
 &
 Irreducible*
 &
 Irreducible*
 &
  Irreducible*
 \\
&&&&&& &\\[-0.2cm]
\hline
\hline
&&&&&&&\\[-0.2cm]
$\begin{pmatrix}
\omega_{11} & \omega_{12} & 0 \\
0 & \omega_{22} & 0 \\
0 & 0 & \omega_{33} \\
\end{pmatrix}$
 &
 Reducible
 &
 Reducible
 &
 Reducible
 &
 Reducible
 &
 Irreducible
 &
 Irreducible*
  &
 Irreducible*
\\
&&&&&&&\\[-0.2cm]
\hline
\hline
&&&&&&&\\[-0.2cm]
$\begin{pmatrix}
\omega_{11} & 0 &  0 \\
\omega_{21} & \omega_{22} & 0 \\
0 & 0 & \omega_{33} \\
\end{pmatrix}$
 &
 Reducible
 &
 Reducible
 &
 Reducible
 &
 Reducible
 &
 Irreducible*
 &
 Irreducible
 &
 Irreducible*
  \\
&&&&&&&\\[-0.2cm]
\hline
\hline
&&&&&&&\\[-0.2cm]
$\begin{pmatrix}
 \omega_{11} & 0 &  \omega_{13} \\
0 & \omega_{22} & 0 \\
0 & 0 & \omega_{33} \\
\end{pmatrix}$
 &
 Reducible
 &
 Reducible
 &
 Reducible
 &
 Irreducible
 &
 Reducible
 &
 Irreducible*
  &
 Irreducible*
 \\
&&&&&&&\\[-0.2cm]
\hline
\hline
&&&&&&&\\[-0.2cm]
$\begin{pmatrix}
\omega_{11} & 0 &  0 \\
0 & \omega_{22} & \omega_{23} \\
0 & 0 & \omega_{33} \\
\end{pmatrix}$
 &
Reducible
 &
 Reducible
 &
 Reducible
 &
 Irreducible
 &
 Irreducible*
 &
 Reducible
  &
 Irreducible*
 \\
&&&&&&&\\[-0.2cm]
\hline
\hline
&&&&&&&\\[-0.2cm]
$\begin{pmatrix}
\omega_{11} & \omega_{12} &  0 \\
\omega_{21} & 0 & 0 \\
0 & 0 & \omega_{33} \\
\end{pmatrix}$
 &
 Reducible
 &
 Reducible
 &
 Reducible
 &
 Reducible
 &
 Irreducible*
 &
 Irreducible*
  &
 Irreducible*
 \\
&&&&&&&\\[-0.2cm]
\hline
\hline
&&&&&&&\\[-0.2cm]
$\begin{pmatrix}
0 & \omega_{12} &  0 \\
\omega_{21} & \omega_{22}  & 0 \\
0 & 0 & \omega_{33} \\
\end{pmatrix}$
 &
Reducible
 &
Reducible
 &
Reducible
 &
Reducible
 &
 Irreducible*
 &
 Irreducible*
 &
 Irreducible* \\
&&&&&&&\\[-0.2cm]
\hline
\hline
&&&&&&&\\[-0.2cm]
$\begin{pmatrix}
0 &  \omega_{12} & \omega_{13} \\
\omega_{21} & 0  & 0 \\
0 & 0 & \omega_{33} \\
\end{pmatrix}$
 &
 Irreducible
 &
 Irreducible
 &
 Irreducible*
 &
 Irreducible
 &
 Irreducible*
 &
 Irreducible*
 &
 Irreducible*
 \\
&&&&&&&\\[-0.2cm]
\hline
\hline
&&&&&&&\\[-0.2cm]
$\begin{pmatrix}
0 & \omega_{12} &  0 \\
\omega_{21} & 0 & \omega_{23} \\
0 & 0 & \omega_{33} \\
\end{pmatrix}$
 &
 Irreducible
 &
 Irreducible
 &
 Irreducible*
 &
 Irreducible
 &
 Irreducible*
 &
 Irreducible*
  &
 Irreducible*
 \\
&&&&&&&\\[-0.2cm]
\hline
\hline
&&&&&&&\\[-0.2cm]
$\begin{pmatrix}
\omega_{11} & \omega_{12} &  \omega_{13} \\
0 & \omega_{22} & 0 \\
0 & 0 & \omega_{33} \\
\end{pmatrix}$
 &
 Irreducible
 &
 Irreducible
 &
 Irreducible
 &
 Irreducible
 &
 Irreducible
 &
 Irreducible*
 &
 Irreducible*
  \\
&&&&&&&\\[-0.2cm]
\hline
\hline
&&&&&&&\\[-0.2cm]
$\begin{pmatrix}
\omega_{11} & 0 &  0 \\
\omega_{21} & \omega_{22} & \omega_{23} \\
0 & 0 & \omega_{33} \\
\end{pmatrix}$
 &
 Irreducible
 &
 Irreducible
 &
 Irreducible
 &
 Irreducible
 &
 Irreducible*
 &
 Irreducible
  &
 Irreducible*
 \\
&&&&&&&\\[-0.2cm]
\hline
\hline
&&&&&&&\\[-0.2cm]
$\begin{pmatrix}
\omega_{11} & \omega_{12} &  0 \\
\omega_{21} & \omega_{22} & 0 \\
0 & 0 & \omega_{33} \\
\end{pmatrix}$
 &
 Reducible
 &
 Reducible
 &
 Reducible
 &
 Reducible
 &
 Irreducible*
 &
 Irreducible*
 &
 Irreducible*
  \\
&&&&&&&\\[-0.2cm]
\hline
\hline
&&&&&&&\\[-0.2cm]
$\begin{pmatrix}
\omega_{11} & \omega_{12} &  0 \\
0 & \omega_{22} & \omega_{23} \\
0 & 0 & \omega_{33} \\
\end{pmatrix}$
 &
 Irreducible
 &
 Irreducible
 &
 Irreducible*
 &
 Irreducible
 &
 Irreducible*
 &
 Irreducible*
 &
 Irreducible* \\
&&&&&&&\\
\hline
\end{tabular}}
\end{center}
\caption{$A$ has a $3$-basic ideal which has a $2$-basic ideal}\label{AntesDeThreeWithTwo}
\end{figure}

\begin{figure}[H]
\begin{center}
\scalebox{0.6}{
\begin{tabular}{|c||c||c||c||c||c||c||c|}
\hline
&&&&&&&\\
$W$ & $U= \tiny{\begin{pmatrix}1 \\0 \\0\end{pmatrix}}$ & $U= \tiny{\begin{pmatrix}0 \\1 \\0\end{pmatrix}}$& $U=\tiny{\begin{pmatrix}0\\0 \\1\end{pmatrix}}$& $U= \tiny{\begin{pmatrix}1 \\1 \\0\end{pmatrix}}$& $U=\tiny{\begin{pmatrix}1 \\0 \\1\end{pmatrix}}$& $U=\tiny{\begin{pmatrix}0 \\1 \\1\end{pmatrix}}$  & $U=\tiny{\begin{pmatrix}1 \\1 \\1\end{pmatrix}}$ \\
&&&&&&&\\[-0.2cm]
\hline
\hline
&&&&&&&\\[-0.2cm]
$\begin{pmatrix}
\omega_{11} & 0 &  \omega_{13} \\
\omega_{21} & \omega_{22} & 0 \\
0 & 0 & \omega_{33} \\
\end{pmatrix}$
 &
 Irreducible
 &
 Irreducible
 &
 Irreducible*
 &
 Irreducible
 &
 Irreducible*
 &
 Irreducible*
 &
 Irreducible*

 \\
&&&&&&&\\[-0.2cm]
\hline
\hline
&&&&&&&\\[-0.2cm]
$\begin{pmatrix}
\omega_{11} & 0 &  \omega_{13} \\
0 & \omega_{22} & \omega_{23} \\
0 & 0 & \omega_{33} \\
\end{pmatrix}$
 &
 Irreducible
 &
 Irreducible
 &
 Irreducible*
 &
 Irreducible
 &
 Irreducible*
 &
 Irreducible*
 &
 Irreducible*

 \\
&&&&&&&\\[-0.2cm]
\hline
\hline
&&&&&&&\\[-0.2cm]
$\begin{pmatrix}
\omega_{11} & \omega_{12} &  \omega_{13} \\
 \omega_{21} & 0 & 0 \\
0 & 0 & \omega_{33} \\
\end{pmatrix}$
 &
 Irreducible
 &
 Irreducible
 &
 Irreducible*
 &
 Irreducible
 &
 Irreducible*
 &
 Irreducible*
 &
 Irreducible*

  \\
&&&&&&&\\[-0.2cm]
\hline
\hline
&&&&&&&\\[-0.2cm]
$\begin{pmatrix}
0 & \omega_{12} &  0 \\
\omega_{21} & \omega_{22} & \omega_{23} \\
0 & 0 & \omega_{33} \\
\end{pmatrix}$
&
 Irreducible
 &
 Irreducible
 &
 Irreducible*
 &
 Irreducible
 &
 Irreducible*
 &
 Irreducible*
 &
 Irreducible*

 \\
&&&&&&&\\[-0.2cm]
\hline
\hline
&&&&&&&\\[-0.2cm]
$\begin{pmatrix}
\omega_{11} & \omega_{12} & 0 \\
\omega_{21} & 0 & \omega_{23}\\
0 & 0 & \omega_{33} \\
\end{pmatrix}$
&
 Irreducible
 &
 Irreducible
 &
 Irreducible*
 &
 Irreducible
 &
 Irreducible*
 &
 Irreducible*
 &
 Irreducible*

 \\
&&&&&&&\\[-0.2cm]
\hline
\hline
&&&&&&&\\[-0.2cm]
$\begin{pmatrix}
0 &  \omega_{12} &  \omega_{13} \\
\omega_{21} & \omega_{22} & 0 \\
0 & 0 & \omega_{33} \\
\end{pmatrix}$
&
 Irreducible
 &
 Irreducible
 &
 Irreducible*
 &
 Irreducible
 &
 Irreducible*
 &
 Irreducible*
 &
 Irreducible*

 \\
&&&&&&&\\[-0.2cm]
\hline
\hline
&&&&&&&\\[-0.2cm]
$\begin{pmatrix}
0 &  \omega_{12} &  \omega_{13} \\
\omega_{21} & 0  & \omega_{23}\\
0 & 0 & \omega_{33} \\
\end{pmatrix}$
&
 Irreducible
 &
 Irreducible
 &
 Irreducible*
 &
 Irreducible
 &
 Irreducible*
 &
 Irreducible*
 &
 Irreducible*

 \\
&&&&&&&\\[-0.2cm]
\hline
\hline
&&&&&&&\\[-0.2cm]
$\begin{pmatrix}
\omega_{11} & \omega_{12} &  \omega_{13} \\
\omega_{21} & \omega_{22} & 0 \\
0 & 0 & \omega_{33} \\
\end{pmatrix}$
&
 Irreducible
 &
 Irreducible
 &
 Irreducible*
 &
 Irreducible
 &
 Irreducible*
 &
 Irreducible*
 &
 Irreducible*

 \\
&&&&&&&\\[-0.2cm]
\hline
\hline
&&&&&&&\\[-0.2cm]
$\begin{pmatrix}
\omega_{11} & \omega_{12} &  0 \\
\omega_{21} & \omega_{22} & \omega_{23} \\
0 & 0 & \omega_{33} \\
\end{pmatrix}$
&
 Irreducible
 &
 Irreducible
 &
 Irreducible*
 &
 Irreducible
 &
 Irreducible*
 &
 Irreducible*
 &
 Irreducible*

 \\
&&&&&&&\\[-0.2cm]
\hline
\hline
&&&&&&&\\[-0.2cm]
$\begin{pmatrix}
\omega_{11} & \omega_{12} &  \omega_{13} \\
0 & \omega_{22} & \omega_{23} \\
0 & 0 & \omega_{33} \\
\end{pmatrix}$
&
 Irreducible
 &
 Irreducible
 &
 Irreducible*
 &
 Irreducible
 &
 Irreducible*
 &
 Irreducible*
 &
 Irreducible*

 \\
&&&&&&&\\[-0.2cm]
\hline
\hline
&&&&&&&\\[-0.2cm]
$\begin{pmatrix}
\omega_{11} & 0 &  \omega_{13} \\
\omega_{21} & \omega_{22} & \omega_{23} \\
0 & 0 & \omega_{33} \\
\end{pmatrix}$
&
 Irreducible
 &
 Irreducible
 &
 Irreducible*
 &
 Irreducible
 &
 Irreducible*
 &
 Irreducible*
 &
 Irreducible*

 \\
&&&&&&&\\[-0.2cm]
\hline
\hline
&&&&&&&\\[-0.2cm]
$\begin{pmatrix}
\omega_{11} & \omega_{12} &  \omega_{13} \\
\omega_{21} & 0 & \omega_{23} \\
0 & 0 & \omega_{33} \\
\end{pmatrix}$
&
 Irreducible
 &
 Irreducible
 &
 Irreducible*
 &
 Irreducible
 &
 Irreducible*
 &
 Irreducible*
 &
 Irreducible*

 \\
&&&&&&&\\[-0.2cm]
\hline
\hline
&&&&&&&\\[-0.2cm]
$\begin{pmatrix}
0 & \omega_{12} &  \omega_{13} \\
\omega_{21} & \omega_{22} & \omega_{23} \\
0 & 0 & \omega_{33} \\
\end{pmatrix}$
&
 Irreducible
 &
 Irreducible
 &
 Irreducible*
 &
 Irreducible
 &
 Irreducible*
 &
 Irreducible*
 &
 Irreducible*

 \\
&&&&&&&\\[-0.2cm]
\hline
\hline
&&&&&&&\\[-0.2cm]
$\begin{pmatrix}
\omega_{11} & \omega_{12} &  \omega_{13} \\
\omega_{21} & \omega_{22} & \omega_{23} \\
0 & 0 & \omega_{33} \\
\end{pmatrix}$
&
 Irreducible
 &
 Irreducible
 &
 Irreducible*
 &
 Irreducible
 &
 Irreducible*
 &
 Irreducible*
 &
 Irreducible*

 \\
&&&&&&&\\
\hline
\end{tabular}}
\end{center}
\caption{$A$ has a 3-basic ideal which has a 2-basic ideal}\label{ThreeWithTwo}
\end{figure}

\smallskip

Now, we analyze whether or not $A$ satisfies Condition (3,2,3).
\smallskip

\subsubsection{$A$ satisfies Condition (3,2,3).}
\medskip

 The following tables describe the different types of mutually non-isomorphic evolution algebras. In order to differentiate the  non-isomorphic algebras we use the following invariants: the number of zeros in the matrix; the number of zeros in the main diagonal (these are invariant as shown in Proposition \ref{BatidoDeFresasConLeche} (i) and (ii)); the indegree and the outdegree of each vertex in the associated graph (which is an invariant as, up to scalars, the only change of basis matrices are the permutations and any permutation preserves these degrees); the graph (since the algebra is perfect, by \cite[Corollary 4.5]{EL}, the graph is an invariant).

 The mutually non-isomorphic evolution algebras are 24. In the table below the types are shown. Two matrices having the same type correspond to isomorphic evolution algebras.

\begin{center}
\begin{figure}[H]
\scalebox{0.60}{
}
\caption{$A$ satisfies Condition (3,2,3)}
\end{figure}
\end{center}

\begin{center}
\begin{figure}[H]
\scalebox{0.60}{
%
}
\caption{$A$ satisfies Condition (3,2,3)}
\end{figure}
\end{center}

\begin{center}
\begin{figure}[H]
\scalebox{0.60}{
%
}
\caption{$A$ satisfies condition (3,2,3)}
\end{figure}
\end{center}

\begin{center}
\begin{figure}[H]
\scalebox{0.60}{
%
}
\caption{$A$ satisfies Condition (3,2,3)}
\end{figure}
\end{center}

\begin{center}
\begin{figure}[H]
\scalebox{0.60}{
%
\right)$
 &
 6
 &
 0
 &
 (2,4)(2,4)(3,1)(3,1)
 &

 \scalebox{0.6}{
$\xymatrix{
  \bullet_{v_1} \ar@/^{-6pt}/[rr] \uloopr{} & &  \bullet_{v_2} \ar@/^{-9pt}/[ll] \uloopr{}    \\
 \bullet_{v_4}     \ar[u] \ar[urr] \dloopr{}  & &  \bullet_{v_3} \dloopr{} \ar[u] \ar[ull]  \\
 }$}
 &
 23
 \\
&&&&&\\
\hline
\end{tabular}}
\caption{$A$ satisfies Condition (3,2,3)}
\end{figure}
\end{center}

\subsubsection{The algebra $A$ does not satisfy Condition (3,2,3).}
\medskip
We explain which is our procedure in order to obtain all the mutually non isomorphic evolution algebras in this subsection.

We are classifying matrices of the form
$$
 (\dag)\quad
 \left(\begin{tabular}{c|c}
$W$
&
$U$
\\
\hline
$\begin{matrix}
0& \cdots & 0
\end{matrix}$
&
 1
\end{tabular}\right),$$
where $W$and $U$ are as in Figures \ref{AntesDeThreeWithTwo} and \ref{ThreeWithTwo}and the matrix $(\dag)$ corresponds to an ``Irreducible*" case, so, we start with  93 cases (37 from Figures \ref{AntesDeThreeWithTwo} and 56 from Figures \ref{ThreeWithTwo}).

Observe that two matrices  of the form
$$
\left(\begin{tabular}{c|c}
$W$
&
$U$
\\
\hline
$\begin{matrix}
0& \cdots & 0
\end{matrix}$
&
 1
\end{tabular}\right),
\quad
\left(\begin{tabular}{c|c}
$W'$
&
$U'$
\\
\hline
$\begin{matrix}
0& \cdots & 0
\end{matrix}$
&
 1
\end{tabular}\right)
$$
correspond to isomorphic evolution algebras if and only if $W$ and $W'$ correspond to three-dimensional isomorphic evolution algebras. The reason is that the possible change of basis matrices are those permutation matrices keeping  the last column invariant.

Choose one matrix as in $(\dag)$ and locate $W$ in \cite[Tables 18-22]{CCY} (which is the classification of three dimensional perfect evolution algebras having a 2-basic ideal generated by the first and the second vectors in the given basis). If there is another $W'$ in the mentioned tables such that the three-dimensional evolution algebras with associated matrices $W$ and $W'$ are isomorphic and such that both algebras have a 2-basic ideal generated by the first and the second vectors in the given basis, then the four-dimensional evolution algebras having these $W$ and $W'$, with the corresponding $U$ and $U'$, are isomorphic.

After doing this, we group the evolution algebras provided in Figures \ref{AntesDeThreeWithTwo} and \ref{ThreeWithTwo} by isomorphisms.

The tables below show which matrices included among the 93 are isomorphic. Two matrices in the same row correspond to isomorphic evolution algebras. Matrices in different rows are not isomorphic.

\newpage

{\begin{figure}[H]
\begin{multicols}{2}
\begin{center}
\scalebox{0.60}{
}

\caption{}\label{Diez}

\columnbreak

\scalebox{0.60}{
	%
}
\caption{}\label{Once}
\end{center}
\end{multicols}
\end{figure}}

{\begin{figure}[H]
		\begin{multicols}{2}
			\begin{center}
				\scalebox{0.60}{
					%
}
\caption{}\label{Doce}

\columnbreak

\scalebox{0.60}{
	%
}
\caption{}\label{Trece}
\end{center}
\end{multicols}
\end{figure}}

{\begin{figure}[H]
		\begin{center}
			\scalebox{0.60}{
				%
}
\caption{}\label{Catorce}
\end{center}
\end{figure}
}

After removing in the set of the 93 the matrices corresponding to isomorphic evolution algebras, the remaining are:

\begin{multicols}{6}
\begin{center}
\scalebox{0.60}{$\left(\begin{tabular}{ccc|c}
$\omega_{11}$ & 0 & 0 &  1\\
0 & $\omega_{22}$  & 0 & 1 \\
0 & 0 & $\omega_{33}$ & 1 \\
\hline
0 & 0 & 0 & 1
\end{tabular}\right)$}
\end{center}

\columnbreak

\begin{center}
\scalebox{0.60}{$\left(\begin{tabular}{ccc|c}
0 & $\omega_{12}$ & 0 &  0\\
$\omega_{21}$ & 0 & 0 & 0 \\
0 & 0 & $\omega_{33}$ & 1 \\
\hline
0 & 0 & 0 & 1
\end{tabular}\right)$}
\end{center}

\columnbreak

\begin{center}\scalebox{0.60}{$\left(\begin{tabular}{ccc|c}
0 & $\omega_{12}$ & 0 & 1 \\
$\omega_{21}$ & 0 & 0 & 1 \\
0 & 0 & $\omega_{33}$ & 1 \\
\hline
0 & 0 & 0 & 1
\end{tabular}\right)$}
\end{center}

\columnbreak

\begin{center}
\scalebox{0.60}{$\left(\begin{tabular}{ccc|c}
$\omega_{11}$ & $\omega_{12}$ & 0 &  0\\
0 & $\omega_{22}$ & 0 & 1 \\
0 & 0 & $\omega_{33}$ & 1 \\
\hline
0 & 0 & 0 & 1
\end{tabular}\right)$}
\end{center}

\columnbreak

\begin{center}\scalebox{0.60}{$\left(\begin{tabular}{ccc|c}
$\omega_{11}$ & $\omega_{12}$ & 0 &  1\\
0 & $\omega_{22}$ & 0 & 1 \\
0 & 0 & $\omega_{33}$ & 1 \\
\hline
0 & 0 & 0 & 1
\end{tabular}\right)$}
\end{center}

\columnbreak

\begin{center}\scalebox{0.60}{$\left(\begin{tabular}{ccc|c}
$\omega_{11}$ &  $\omega_{12}$ & 0 & 1 \\
 $\omega_{21}$ & 0 & 0 & 0 \\
0 & 0 & $\omega_{33}$ & 1 \\
\hline
0 & 0 & 0 & 1
\end{tabular}\right)$}
\end{center}

\end{multicols}

\begin{multicols}{6}

\begin{center}
\scalebox{0.60}{$\left(\begin{tabular}{ccc|c}
$\omega_{11}$ &  $\omega_{12}$ & 0 & 0 \\
 $\omega_{21}$ & 0 & 0 & 1 \\
0 & 0 & $\omega_{33}$ & 1 \\
\hline
0 & 0 & 0 & 1
\end{tabular}\right)$}
\end{center}

\columnbreak

\begin{center}\scalebox{0.60}{$\left(\begin{tabular}{ccc|c}
$\omega_{11}$ &  $\omega_{12}$ & 0 & 1 \\
$\omega_{21}$ & 0 & 0   & 1 \\
0 & 0 & $\omega_{33}$ & 1 \\
\hline
0 & 0 & 0 & 1
\end{tabular}\right)$}
\end{center}

\columnbreak

\begin{center}\scalebox{0.60}{$\left(\begin{tabular}{ccc|c}
0 &  $\omega_{12}$ & $\omega_{13}$ & 0 \\
 $\omega_{21}$ &  0 & 0 & 0 \\
0 & 0 & $\omega_{33}$ & 1 \\
\hline
0 & 0 & 0 & 1
\end{tabular}\right)$}
\end{center}

\columnbreak

\begin{center}\scalebox{0.60}{$\left(\begin{tabular}{ccc|c}
0 &  $\omega_{12}$ & $\omega_{13}$ & 1 \\
 $\omega_{21}$ &  0 & 0 & 0 \\
0 & 0 & $\omega_{33}$ & 1 \\
\hline
0 & 0 & 0 & 1
\end{tabular}\right)$}
\end{center}

\columnbreak

\begin{center}
\scalebox{0.60}{$\left(\begin{tabular}{ccc|c}
0 &  $\omega_{12}$ & $\omega_{13}$ & 0 \\
 $\omega_{21}$ &  0 & 0 & 1 \\
0 & 0 & $\omega_{33}$ & 1 \\
\hline
0 & 0 & 0 & 1
\end{tabular}\right)$}
\end{center}

\columnbreak

\begin{center}\scalebox{0.60}{$\left(\begin{tabular}{ccc|c}
0 &  $\omega_{12}$ & $\omega_{13}$ & 1 \\
 $\omega_{21}$ &  0 & 0 & 1 \\
0 & 0 & $\omega_{33}$ & 1 \\
\hline
0 & 0 & 0 & 1
\end{tabular}\right)$}
\end{center}

\end{multicols}

\begin{multicols}{6}

\begin{center}\scalebox{0.60}{$\left(\begin{tabular}{ccc|c}
$\omega_{11}$ &  $\omega_{12}$ & $\omega_{13}$ & 0 \\
0 &  $\omega_{22}$ & 0  & 1 \\
0 & 0 & $\omega_{33}$ & 1 \\
\hline
0 & 0 & 0 & 1
\end{tabular}\right)$}
\end{center}

\columnbreak

\begin{center}\scalebox{0.60}{$\left(\begin{tabular}{ccc|c}
$\omega_{11}$ &  $\omega_{12}$ & $\omega_{13}$ & 1 \\
0 &  $\omega_{22}$ & 0  & 1 \\
0 & 0 & $\omega_{33}$ & 1 \\
\hline
0 & 0 & 0 & 1
\end{tabular}\right)$}
\end{center}

\columnbreak

\begin{center}\scalebox{0.60}{$\left(\begin{tabular}{ccc|c}
$\omega_{11}$ &  $\omega_{12}$ & 0 & 1 \\
$\omega_{21}$ &  $\omega_{22}$ & 0  & 0 \\
0 & 0 & $\omega_{33}$ & 1 \\
\hline
0 & 0 & 0 & 1
\end{tabular}\right)$}
\end{center}

\columnbreak

\begin{center}\scalebox{0.60}{$\left(\begin{tabular}{ccc|c}
$\omega_{11}$ &  $\omega_{12}$ & 0 & 1 \\
$\omega_{21}$ &  $\omega_{22}$ & 0  & 1 \\
0 & 0 & $\omega_{33}$ & 1 \\
\hline
0 & 0 & 0 & 1
\end{tabular}\right)$}
\end{center}

\columnbreak

\begin{center}\scalebox{0.60}{$\left(\begin{tabular}{ccc|c}
$\omega_{11}$ &  $\omega_{12}$ & 0 & 0 \\
0 &  $\omega_{22}$ & $\omega_{23}$  & 0 \\
0 & 0 & $\omega_{33}$ & 1 \\
\hline
0 & 0 & 0 & 1
\end{tabular}\right)$}
\end{center}

\columnbreak

\begin{center}
\scalebox{0.60}{$\left(\begin{tabular}{ccc|c}
$\omega_{11}$ &  $\omega_{12}$ & 0 & 1 \\
0 &  $\omega_{22}$ & $\omega_{23}$  & 0 \\
0 & 0 & $\omega_{33}$ & 1 \\
\hline
0 & 0 & 0 & 1
\end{tabular}\right)$}
\end{center}

\end{multicols}

\begin{multicols}{6}

\begin{center}\scalebox{0.60}{$\left(\begin{tabular}{ccc|c}
$\omega_{11}$ &  $\omega_{12}$ & 0 & 0 \\
0 &  $\omega_{22}$ & $\omega_{23}$  & 1 \\
0 & 0 & $\omega_{33}$ & 1 \\
\hline
0 & 0 & 0 & 1
\end{tabular}\right)$}
\end{center}

\columnbreak

\begin{center}\scalebox{0.60}{$\left(\begin{tabular}{ccc|c}
$\omega_{11}$ &  $\omega_{12}$ & 0 & 1 \\
0 &  $\omega_{22}$ & $\omega_{23}$  & 1 \\
0 & 0 & $\omega_{33}$ & 1 \\
\hline
0 & 0 & 0 & 1
\end{tabular}\right)$}
\end{center}

\columnbreak

\begin{center}\scalebox{0.60}{$\left(\begin{tabular}{ccc|c}
$\omega_{11}$ &  0 & $\omega_{13}$ & 0 \\
0 &  $\omega_{22}$ & $\omega_{23}$  & 0 \\
0 & 0 & $\omega_{33}$ & 1 \\
\hline
0 & 0 & 0 & 1
\end{tabular}\right)$}
\end{center}

\columnbreak

\begin{center}\scalebox{0.60}{$\left(\begin{tabular}{ccc|c}
		$\omega_{11}$ &  0 & $\omega_{13}$ & 1 \\
		0 &  $\omega_{22}$ & $\omega_{23}$  & 0 \\
		0 & 0 & $\omega_{33}$ & 1 \\
		\hline
		0 & 0 & 0 & 1
		\end{tabular}\right)$}
\end{center}

\columnbreak

\begin{center}
\scalebox{0.60}{$\left(\begin{tabular}{ccc|c}
$\omega_{11}$ &  0 & $\omega_{13}$ & 1 \\
0 &  $\omega_{22}$ & $\omega_{23}$  & 1 \\
0 & 0 & $\omega_{33}$ & 1 \\
\hline
0 & 0 & 0 & 1
\end{tabular}\right)$}
\end{center}

\columnbreak

\begin{center}\scalebox{0.60}{$\left(\begin{tabular}{ccc|c}
$\omega_{11}$ &  $\omega_{12}$ & $\omega_{13}$ & 0 \\
$\omega_{21}$ & 0 & 0  & 0 \\
0 & 0 & $\omega_{33}$ & 1 \\
\hline
0 & 0 & 0 & 1
\end{tabular}\right)$}
\end{center}

\end{multicols}

\begin{multicols}{6}

\begin{center}\scalebox{0.60}{$\left(\begin{tabular}{ccc|c}
$\omega_{11}$ &  $\omega_{12}$ & $\omega_{13}$ & 1 \\
$\omega_{21}$ & 0 & 0  & 0 \\
0 & 0 & $\omega_{33}$ & 1 \\
\hline
0 & 0 & 0 & 1
\end{tabular}\right)$}
\end{center}

\columnbreak

\begin{center}\scalebox{0.60}{$\left(\begin{tabular}{ccc|c}
$\omega_{11}$ &  $\omega_{12}$ & $\omega_{13}$ & 0 \\
$\omega_{21}$ & 0 & 0  & 1 \\
0 & 0 & $\omega_{33}$ & 1 \\
\hline
0 & 0 & 0 & 1
\end{tabular}\right)$}
\end{center}

\columnbreak

\begin{center}
\scalebox{0.60}{$\left(\begin{tabular}{ccc|c}
$\omega_{11}$ &  $\omega_{12}$ & $\omega_{13}$ & 1 \\
$\omega_{21}$ & 0 & 0  & 1 \\
0 & 0 & $\omega_{33}$ & 1 \\
\hline
0 & 0 & 0 & 1
\end{tabular}\right)$}
\end{center}

\columnbreak

\begin{center}\scalebox{0.60}{$\left(\begin{tabular}{ccc|c}
$\omega_{11}$ &  $\omega_{12}$ & $\omega_{13}$ & 0 \\
0 &  $\omega_{22}$ & 0  & 1 \\
0 & 0 & $\omega_{33}$ & 1 \\
\hline
0 & 0 & 0 & 1
\end{tabular}\right)$}
\end{center}

\columnbreak

\begin{center}\scalebox{0.60}{$\left(\begin{tabular}{ccc|c}
$\omega_{11}$ &  $\omega_{12}$ & 0 & 0 \\
$\omega_{21}$ & 0 & $\omega_{23}$ & 0 \\
0 & 0 & $\omega_{33}$ & 1 \\
\hline
0 & 0 & 0 & 1
\end{tabular}\right)$}
\end{center}

\columnbreak

\begin{center}\scalebox{0.60}{$\left(\begin{tabular}{ccc|c}
$\omega_{11}$ &  $\omega_{12}$ & 0 & 1 \\
$\omega_{21}$ & 0 & $\omega_{23}$ & 0 \\
0 & 0 & $\omega_{33}$ & 1 \\
\hline
0 & 0 & 0 & 1
\end{tabular}\right)$}
\end{center}

\end{multicols}

\begin{multicols}{6}

\begin{center}
\scalebox{0.60}{$\left(\begin{tabular}{ccc|c}
$\omega_{11}$ &  $\omega_{12}$ & 0 & 0 \\
$\omega_{21}$ & 0 & $\omega_{23}$ & 1 \\
0 & 0 & $\omega_{33}$ & 1 \\
\hline
0 & 0 & 0 & 1
\end{tabular}\right)$}
\end{center}

\columnbreak

\begin{center}\scalebox{0.60}{$\left(\begin{tabular}{ccc|c}
$\omega_{11}$ &  $\omega_{12}$ & 0 & 1 \\
$\omega_{21}$ & 0 & $\omega_{23}$ & 1 \\
0 & 0 & $\omega_{33}$ & 1 \\
\hline
0 & 0 & 0 & 1
\end{tabular}\right)$}
\end{center}

\columnbreak

\begin{center}\scalebox{0.60}{$\left(\begin{tabular}{ccc|c}
0 &  $\omega_{12}$ & $\omega_{13}$ & 0 \\
$\omega_{21}$  & 0 &   $\omega_{23}$   & 0 \\
0 & 0 & $\omega_{33}$ & 1 \\
\hline
0 & 0 & 0 & 1
\end{tabular}\right)$}
\end{center}

\columnbreak

\begin{center}\scalebox{0.60}{$\left(\begin{tabular}{ccc|c}
0 &  $\omega_{12}$ & $\omega_{13}$ & 1 \\
$\omega_{21}$  & 0 &   $\omega_{23}$   & 0 \\
0 & 0 & $\omega_{33}$ & 1 \\
\hline
0 & 0 & 0 & 1
\end{tabular}\right)$}
\end{center}

\columnbreak

\begin{center}
\scalebox{0.60}{$\left(\begin{tabular}{ccc|c}
0 &  $\omega_{12}$ & $\omega_{13}$ & 1 \\
$\omega_{21}$  & 0 &   $\omega_{23}$   & 1 \\
0 & 0 & $\omega_{33}$ & 1 \\
\hline
0 & 0 & 0 & 1
\end{tabular}\right)$}
\end{center}

\columnbreak

\begin{center}\scalebox{0.60}{$\left(\begin{tabular}{ccc|c}
$\omega_{11}$ &  $\omega_{12}$ & $\omega_{13}$ & 0 \\
$\omega_{21}$ &  $\omega_{22}$ & 0  & 0 \\
0 & 0 & $\omega_{33}$ & 1 \\
\hline
0 & 0 & 0 & 1
\end{tabular}\right)$}
\end{center}

\end{multicols}

\begin{multicols}{6}

\begin{center}\scalebox{0.60}{$\left(\begin{tabular}{ccc|c}
$\omega_{11}$ &  $\omega_{12}$ & $\omega_{13}$ & 1 \\
$\omega_{21}$ &  $\omega_{22}$ & 0  & 0 \\
0 & 0 & $\omega_{33}$ & 1 \\
\hline
0 & 0 & 0 & 1
\end{tabular}\right)$}
\end{center}

\columnbreak

\begin{center}\scalebox{0.60}{$\left(\begin{tabular}{ccc|c}
$\omega_{11}$ &  $\omega_{12}$ & $\omega_{13}$ & 0 \\
$\omega_{21}$ &  $\omega_{22}$ & 0  & 1 \\
0 & 0 & $\omega_{33}$ & 1 \\
\hline
0 & 0 & 0 & 1
\end{tabular}\right)$}
\end{center}

\columnbreak

\begin{center}
\scalebox{0.60}{$\left(\begin{tabular}{ccc|c}
$\omega_{11}$ &  $\omega_{12}$ & $\omega_{13}$ & 1 \\
$\omega_{21}$ &  $\omega_{22}$ & 0  & 1 \\
0 & 0 & $\omega_{33}$ & 1 \\
\hline
0 & 0 & 0 & 1
\end{tabular}\right)$}
\end{center}

\columnbreak

\begin{center}\scalebox{0.60}{$\left(\begin{tabular}{ccc|c}
$\omega_{11}$ &  $\omega_{12}$ & $\omega_{13}$ & 0 \\
0 &  $\omega_{22}$ & $\omega_{23}$  & 0 \\
0 & 0 & $\omega_{33}$ & 1 \\
\hline
0 & 0 & 0 & 1
\end{tabular}\right)$}
\end{center}

\columnbreak

\begin{center}\scalebox{0.60}{$\left(\begin{tabular}{ccc|c}
$\omega_{11}$ &  $\omega_{12}$ & $\omega_{13}$ & 1 \\
0 &  $\omega_{22}$ & $\omega_{23}$  & 0 \\
0 & 0 & $\omega_{33}$ & 1 \\
\hline
0 & 0 & 0 & 1
\end{tabular}\right)$}
\end{center}

\columnbreak

\begin{center}\scalebox{0.60}{$\left(\begin{tabular}{ccc|c}
$\omega_{11}$ &  $\omega_{12}$ & $\omega_{13}$ & 0 \\
0 &  $\omega_{22}$ & $\omega_{23}$  & 1 \\
0 & 0 & $\omega_{33}$ & 1 \\
\hline
0 & 0 & 0 & 1
\end{tabular}\right)$}
\end{center}

\end{multicols}

\begin{multicols}{6}

\begin{center}
\scalebox{0.60}{$\left(\begin{tabular}{ccc|c}
$\omega_{11}$ &  $\omega_{12}$ & $\omega_{13}$ & 1 \\
0 &  $\omega_{22}$ & $\omega_{23}$  & 1 \\
0 & 0 & $\omega_{33}$ & 1 \\
\hline
0 & 0 & 0 & 1
\end{tabular}\right)$}
\end{center}

\columnbreak

\begin{center}\scalebox{0.60}{$\left(\begin{tabular}{ccc|c}
$\omega_{11}$ &  $\omega_{12}$ & $\omega_{13}$ & 0 \\
$\omega_{21}$ &  0 & $\omega_{23}$  & 0 \\
0 & 0 & $\omega_{33}$ & 1 \\
\hline
0 & 0 & 0 & 1
\end{tabular}\right)$}
\end{center}

\columnbreak

\begin{center}\scalebox{0.60}{$\left(\begin{tabular}{ccc|c}
$\omega_{11}$ &  $\omega_{12}$ & $\omega_{13}$ & 1 \\
$\omega_{21}$ &  0 & $\omega_{23}$  & 0 \\
0 & 0 & $\omega_{33}$ & 1 \\
\hline
0 & 0 & 0 & 1
\end{tabular}\right)$}
\end{center}

\columnbreak

\begin{center}\scalebox{0.60}{$\left(\begin{tabular}{ccc|c}
$\omega_{11}$ &  $\omega_{12}$ & $\omega_{13}$ & 0 \\
$\omega_{21}$ &  0 & $\omega_{23}$  & 1 \\
0 & 0 & $\omega_{33}$ & 1 \\
\hline
0 & 0 & 0 & 1
\end{tabular}\right)$}
\end{center}

\columnbreak

\begin{center}
\scalebox{0.60}{$\left(\begin{tabular}{ccc|c}
$\omega_{11}$ &  $\omega_{12}$ & $\omega_{13}$ & 1 \\
$\omega_{21}$ &  0 & $\omega_{23}$  & 1 \\
0 & 0 & $\omega_{33}$ & 1 \\
\hline
0 & 0 & 0 & 1
\end{tabular}\right)$}
\end{center}

\columnbreak

\begin{center}\scalebox{0.60}{$\left(\begin{tabular}{ccc|c}
$\omega_{11}$ &  $\omega_{12}$ & $\omega_{13}$ & 0 \\
$\omega_{21}$ &  $\omega_{22}$ & $\omega_{23}$  & 0 \\
0 & 0 & $\omega_{33}$ & 1 \\
\hline
0 & 0 & 0 & 1
\end{tabular}\right)$}
\end{center}

\end{multicols}

\begin{multicols}{6}

\begin{center}\scalebox{0.60}{$\left(\begin{tabular}{ccc|c}
$\omega_{11}$ &  $\omega_{12}$ & $\omega_{13}$ & 1 \\
$\omega_{21}$ &  $\omega_{22}$ & $\omega_{23}$  & 0 \\
0 & 0 & $\omega_{33}$ & 1 \\
\hline
0 & 0 & 0 & 1
\end{tabular}\right)$}
\end{center}

\columnbreak

\begin{center}\scalebox{0.60}{$\left(\begin{tabular}{ccc|c}
$\omega_{11}$ &  $\omega_{12}$ & $\omega_{13}$ & 1 \\
$\omega_{21}$ &  $\omega_{22}$ & $\omega_{23}$  & 1 \\
0 & 0 & $\omega_{33}$ & 1 \\
\hline
0 & 0 & 0 & 1
\end{tabular}\right)$}
\end{center}

\end{multicols}

\subsection{No 3-basic ideal of $A$ has a 2-basic ideal.}
\subsubsection{The matrix $U$ has one nonzero entry}\label{Alp1}

The possible matrices are of the form:
\begin{center}
\scalebox{0.60}{$
%
,
$}
\end{center}
where the entries in the last column are nonzero. It happens that the three of them produce isomorphic evolution algebras and the change of basis matrices are given by $I_{(2,3)}$ (producing the isomorphism between the first and the second) and by $I_{(1,3)}$ (producing the isomorphism between the first and the third). Moreover, by multiplying conveniently one of the vectors in the basis by an scalar, we may assume that the matrix is:
\begin{equation}\label{bocadillito1}
\scalebox{0.60}{$
%
$}
\end{equation}
In order to get an irredundant classification we may keep the last row and the last column as in \eqref{bocadillito1} (use what has been explained above and use Proposition \ref{OneDimensional}). This implies that the only possible change of basis which is allowed is $I_{(1,2)}$. Moreover, the matrix $W$ corresponds to a three-dimensional evolution algebra not having a 2-basic ideal.
Therefore, in the tables that follow we have inserted, as matrix $W$, all the different matrices associated to three-dimensional evolution algebras not having a 2-basic ideal (these appear in the classification in  \cite{CCY, CSV2}).

In the tables that follow, matrices in the same row corresponds to isomorphic evolution algebras. Matrices in different rows are associated to non-isomorphic evolution algebras.

\medskip

\begin{figure}[H]
\begin{multicols}{2}
\centering
\scalebox{0.6}{
}
\caption{}\label{Catorce}
\columnbreak

\scalebox{0.6}{
%
}
\caption{}\label{diecinueve}
\end{multicols}
\end{figure}

\begin{figure}[H]
\begin{multicols}{2}
\scalebox{0.600}{
%
}
\caption{}\label{dieciocho}

\columnbreak

\scalebox{0.600}{

%
}
\caption{}\label{Catorce}
\end{multicols}
\end{figure}

\newpage

\begin{figure}[H]
\begin{multicols}{2}

\scalebox{0.600}{

%
}
\caption{}\label{veinte}

\columnbreak

\scalebox{0.600}{

%
}
\caption{}\label{fresquito8}

\end{multicols}
\end{figure}

\newpage

\begin{figure}[H]
\begin{multicols}{2}
\scalebox{0.600}{
%
}
\caption{}\label{veintiuno}

\columnbreak
\scalebox{0.600}{

%
}
\caption{}\label{veintidos}

\end{multicols}
\end{figure}

\newpage

\begin{figure}[H]
\begin{multicols}{2}
\scalebox{0.600}{
%
}
\caption{}\label{veintitres}

\columnbreak
\scalebox{0.600}{

%
}
\caption{}\label{veinticuatro}
\end{multicols}
\end{figure}

\subsubsection{The matrix $U$ has two nonzero entries}

We are reasoning in a similar way as in Case \ref{Alp1}. Here, the possible matrices are of the form:
\medskip

\begin{center}
\scalebox{0.60}{
$
%
.
$}
\end{center}
\medskip

All of them produce isomorphic evolution algebras and the change of basis matrices are given by $I_{(2,3)}$ (producing the isomorphism between the first and the second) and by $I_{(1,3)}$ (producing the isomorphism between the first and the third). Moreover, by multiplying conveniently one of the vectors in the basis by an scalar, we may assume that the matrix is:

\begin{equation}\label{churritoSiete}
\scalebox{0.60}{$
%
$}
\end{equation}
In the tables that follow we include the possible matrices that appear. All the matrices in the second column of each table correspond to isomorphic evolution algebras in which the isomorphism is the exchange of the first and the second vectors in the natural basis.

\begin{figure}[H]
\begin{multicols}{2}
\begin{center}
\scalebox{0.60}{
%
}
\caption{}\label{veinticuatro1}

\columnbreak

\scalebox{0.60}{
%
}
\caption{}\label{fresquito10}
\end{center}
\end{multicols}
\end{figure}

\begin{figure}[H]
\begin{multicols}{2}
\scalebox{0.600}{
%
}
\caption{}\label{veinticuatro2}

\columnbreak

\scalebox{0.600}{

%
}
\caption{}\label{veinticuatro3}

\end{multicols}
\end{figure}

\newpage

\begin{figure}[H]
\begin{multicols}{2}

\scalebox{0.600}{

%
}
\caption{}\label{veinticuatro4}

\columnbreak

\scalebox{0.600}{

%
}
\caption{}\label{veinticuatro5}

\end{multicols}
\end{figure}

\newpage

\begin{figure}[H]
\begin{multicols}{2}
\scalebox{0.600}{
%
}
\caption{}\label{treinta}

\end{multicols}
\end{figure}

\newpage

\begin{figure}[H]
\begin{multicols}{2}
\scalebox{0.600}{
%
}
\caption{}\label{cuarenta}

\columnbreak
\scalebox{0.600}{

%
\right)$
\\
&\\
\hline
\end{tabular}}
\caption{}\label{quinientos}
\end{multicols}
\end{figure}

\subsubsection{The matrix $U$ has no nonzero entries.}
\medskip

If we analyze which are the change of basis matrices that are allowed we get that
$I_\sigma M $ is, up to scalars, of the same form as $M$ if $\sigma$ belongs to the group $G:= \{{\rm id}, (2,3), (1,2), (1, 2,3), (1,3,2), (1, 3)\}$, which consists of all permutations in $S_4$ which leave $4$  invariant.

In order to classify we proceed by choosing the matrices $Y$ as all the different matrices associated to  three-dimensional evolution algebras not having a 2-basic ideal which appear in the classification in \cite{CCY, CSV2}. The mutually non-isomorphic matrices are listed below.

\begin{multicols}{6}
\begin{center}
\scalebox{0.60}{
$\left(\begin{tabular}{ccc|c}
0 & 0 & $\omega_{13}$ &  1\\
$\omega_{21}$ & 0 & 0 & 1 \\
0 & $\omega_{32}$ & 0 & 1 \\
\hline
0 & 0 & 0 & 1
\end{tabular}\right)$}
\end{center}

\columnbreak

\begin{center}
	\scalebox{0.60}{
$\left(\begin{tabular}{ccc|c}
0 & $\omega_{12}$ & 0 &  1\\
$\omega_{21}$ & 0 & 0 & 1 \\
$\omega_{31}$ & 0 & $\omega_{33}$ & 1 \\
\hline
0 & 0 & 0 & 1
\end{tabular}\right)$}
\end{center}

\columnbreak

\begin{center}
	\scalebox{0.60}{
$\left(\begin{tabular}{ccc|c}
$\omega_{11}$ & 0 & $\omega_{13}$ &  1\\
$\omega_{21}$ & 0 & 0 & 1 \\
0 & $\omega_{32}$ & 0 & 1 \\
\hline
0 & 0 & 0 & 1
\end{tabular}\right)$}
\end{center}

\columnbreak

\begin{center}
	\scalebox{0.60}{
$\left(\begin{tabular}{ccc|c}
0 & $\omega_{12}$ & $\omega_{13}$ &  1\\
$\omega_{21}$ & 0 & 0 & 1 \\
0 & $\omega_{32}$ & 0 & 1 \\
\hline
0 & 0 & 0 & 1
\end{tabular}\right)$}
\end{center}

\columnbreak

\begin{center}
	\scalebox{0.60}{
$\left(\begin{tabular}{ccc|c}
$\omega_{11}$ & $\omega_{12}$ & 0 &  1\\
$\omega_{21}$ & 0 & 0 & 1 \\
$\omega_{31}$ & 0 & $\omega_{33}$ & 1 \\
\hline
0 & 0 & 0 & 1
\end{tabular}\right)$}
\end{center}

\columnbreak

\begin{center}
	\scalebox{0.60}{
$\left(\begin{tabular}{ccc|c}
$\omega_{11}$ & $\omega_{12}$ & 0 &  1\\
$\omega_{21}$ & 0 & 0 & 1 \\
0 & $\omega_{32}$ & $\omega_{33}$ & 1 \\
\hline
0 & 0 & 0 & 1
\end{tabular}\right)$}
\end{center}

\end{multicols}

\begin{multicols}{6}

\begin{center}
	\scalebox{0.60}{
$\left(\begin{tabular}{ccc|c}
0 & $\omega_{12}$ & 0 &  1\\
$\omega_{21}$ & 0 & $\omega_{23}$ & 1 \\
$\omega_{31}$ & 0 & $\omega_{33}$ & 1 \\
\hline
0 & 0 & 0 & 1
\end{tabular}\right)$}
\end{center}

\columnbreak

\begin{center}
	\scalebox{0.60}{
$\left(\begin{tabular}{ccc|c}
0 & $\omega_{12}$ & 0 &  1\\
$\omega_{21}$ & 0 & $\omega_{23}$ & 1 \\
0 & $\omega_{32}$ & $\omega_{33}$ & 1 \\
\hline
0 & 0 & 0 & 1
\end{tabular}\right)$}
\end{center}

\columnbreak

\begin{center}
	\scalebox{0.60}{
$\left(\begin{tabular}{ccc|c}
0 & $\omega_{12}$ & 0 &  1\\
$\omega_{21}$ & 0 & 0 & 1 \\
$\omega_{31}$ & $\omega_{32}$ & $\omega_{33}$ & 1 \\
\hline
0 & 0 & 0 & 1
\end{tabular}\right)$}
\end{center}

\columnbreak

\begin{center}
	\scalebox{0.60}{
$\left(\begin{tabular}{ccc|c}
$\omega_{11}$ & $\omega_{12}$ & $\omega_{13}$ &  1\\
$\omega_{21}$ & 0 & 0 & 1 \\
0 & $\omega_{32}$ & 0 & 1 \\
\hline
0 & 0 & 0 & 1
\end{tabular}\right)$}
\end{center}

\columnbreak

\begin{center}
	\scalebox{0.60}{
$\left(\begin{tabular}{ccc|c}
$\omega_{11}$ & 0 & $\omega_{13}$ &  1\\
$\omega_{21}$ & $\omega_{22}$ & 0 & 1 \\
0 & $\omega_{32}$ & 0 & 1 \\
\hline
0 & 0 & 0 & 1
\end{tabular}\right)$}
\end{center}

\columnbreak

\begin{center}
	\scalebox{0.60}{
$\left(\begin{tabular}{ccc|c}
$\omega_{11}$ & 0 & $\omega_{13}$ &  1\\
$\omega_{21}$ & 0 & 0 & 1 \\
$\omega_{31}$ & $\omega_{32}$ & 0 & 1 \\
\hline
0 & 0 & 0 & 1
\end{tabular}\right)$}
\end{center}
\end{multicols}

\begin{multicols}{6}
\begin{center}
	\scalebox{0.60}{
$\left(\begin{tabular}{ccc|c}
0 & $\omega_{12}$ & $\omega_{13}$ &  1\\
$\omega_{21}$ & 0 & $\omega_{23}$ & 1 \\
0 & $\omega_{32}$ & 0 & 1 \\
\hline
0 & 0 & 0 & 1
\end{tabular}\right)$}
\end{center}

\columnbreak
\begin{center}
	\scalebox{0.60}{
$\left(\begin{tabular}{ccc|c}
$\omega_{11}$ & $\omega_{12}$ & 0 &  1\\
0 & $\omega_{22}$ & $\omega_{23}$ & 1 \\
$\omega_{31}$ & 0 & $\omega_{33}$ & 1 \\
\hline
0 & 0 & 0 & 1
\end{tabular}\right)$}
\end{center}

\columnbreak

\begin{center}
	\scalebox{0.60}{
$\left(\begin{tabular}{ccc|c}
$\omega_{11}$ & $\omega_{12}$ & 0 &  1\\
$\omega_{21}$ & $\omega_{22}$ & 0 & 1 \\
$\omega_{31}$ & 0 & $\omega_{33}$ & 1 \\
\hline
0 & 0 & 0 & 1
\end{tabular}\right)$}
\end{center}

\columnbreak

\begin{center}
	\scalebox{0.60}{
$\left(\begin{tabular}{ccc|c}
$\omega_{11}$ & $\omega_{12}$ & $\omega_{13}$ &  1\\
$\omega_{21}$ & 0 & 0 & 1 \\
$\omega_{31}$ & 0 & $\omega_{33}$ & 1 \\
\hline
0 & 0 & 0 & 1
\end{tabular}\right)$}
\end{center}

\columnbreak

\begin{center}
	\scalebox{0.60}{
$\left(\begin{tabular}{ccc|c}
$\omega_{11}$ & $\omega_{12}$ & $\omega_{13}$ &  1\\
$\omega_{21}$ & 0 & 0 & 1 \\
0 & $\omega_{32}$ & $\omega_{33}$ & 1 \\
\hline
0 & 0 & 0 & 1
\end{tabular}\right)$}
\end{center}

\columnbreak

\begin{center}
	\scalebox{0.60}{
$\left(\begin{tabular}{ccc|c}
0 & $\omega_{12}$ & $\omega_{13}$ &  1\\
$\omega_{21}$ & $\omega_{22}$ & 0 & 1 \\
$\omega_{31}$ & 0 & $\omega_{33}$ & 1 \\
\hline
0 & 0 & 0 & 1
\end{tabular}\right)$}
\end{center}

\end{multicols}

\begin{multicols}{6}

\begin{center}
	\scalebox{0.60}{
$\left(\begin{tabular}{ccc|c}
0 & $\omega_{12}$ & $\omega_{13}$ &  1\\
$\omega_{21}$ & $\omega_{22}$ & 0 & 1 \\
0 & $\omega_{32}$ & $\omega_{33}$ & 1 \\
\hline
0 & 0 & 0 & 1
\end{tabular}\right)$}
\end{center}

\columnbreak

\begin{center}
	\scalebox{0.60}{
$\left(\begin{tabular}{ccc|c}
0 & $\omega_{12}$ & $\omega_{13}$ &  1\\
$\omega_{21}$ & 0 & $\omega_{23}$ & 1 \\
$\omega_{31}$ & 0 & $\omega_{33}$ & 1 \\
\hline
0 & 0 & 0 & 1
\end{tabular}\right)$}
\end{center}

\columnbreak

\begin{center}
	\scalebox{0.60}{
$\left(\begin{tabular}{ccc|c}
0 & $\omega_{12}$ & $\omega_{13}$ &  1\\
$\omega_{21}$ & 0 & 0 & 1 \\
$\omega_{31}$ & $\omega_{32}$ & $\omega_{33}$ & 1 \\
\hline
0 & 0 & 0 & 1
\end{tabular}\right)$}
\end{center}

\columnbreak

\begin{center}
	\scalebox{0.60}{
$\left(\begin{tabular}{ccc|c}
0 & $\omega_{12}$ & 0 &  1\\
$\omega_{21}$ & $\omega_{22}$ & 0 & 1 \\
$\omega_{31}$ & $\omega_{32}$ & $\omega_{33}$ & 1 \\
\hline
0 & 0 & 0 & 1
\end{tabular}\right)$}
\end{center}

\columnbreak

\begin{center}
	\scalebox{0.60}{
$\left(\begin{tabular}{ccc|c}
$\omega_{11}$ & $\omega_{12}$ & $\omega_{13}$ &  1\\
$\omega_{21}$ & $\omega_{22}$ & 0 & 1 \\
0 & $\omega_{32}$ & 0 & 1 \\
\hline
0 & 0 & 0 & 1
\end{tabular}\right)$}
\end{center}

\columnbreak

\begin{center}
	\scalebox{0.60}{
$\left(\begin{tabular}{ccc|c}
$\omega_{11}$ & $\omega_{12}$ & $\omega_{13}$ &  1\\
$\omega_{21}$ & 0 & 0 & 1 \\
$\omega_{31}$ & $\omega_{32}$ & 0 & 1 \\
\hline
0 & 0 & 0 & 1
\end{tabular}\right)$}
\end{center}
\end{multicols}

\begin{multicols}{6}
\begin{center}
	\scalebox{0.60}{
$\left(\begin{tabular}{ccc|c}
0 & $\omega_{12}$ & $\omega_{13}$ &  1\\
$\omega_{21}$ & 0 & $\omega_{23}$ & 1 \\
$\omega_{31}$ & $\omega_{32}$ & 0 & 1 \\
\hline
0 & 0 & 0 & 1
\end{tabular}\right)$}
\end{center}

\columnbreak


\begin{center}
	\scalebox{0.60}{
$\left(\begin{tabular}{ccc|c}
$\omega_{11}$ & $\omega_{12}$ & $\omega_{13}$ &  1\\
$\omega_{21}$ & $\omega_{22}$ & 0 & 1 \\
$\omega_{31}$ & 0 & $\omega_{33}$ & 1 \\
\hline
0 & 0 & 0 & 1
\end{tabular}\right)$}
\end{center}

\columnbreak

\begin{center}
	\scalebox{0.60}{
$\left(\begin{tabular}{ccc|c}
$\omega_{11}$ & $\omega_{12}$ & $\omega_{13}$ &  1\\
$\omega_{21}$ & $\omega_{22}$ & 0 & 1 \\
0 & $\omega_{32}$ & $\omega_{33}$ & 1 \\
\hline
0 & 0 & 0 & 1
\end{tabular}\right)$}
\end{center}

\columnbreak

\begin{center}
	\scalebox{0.60}{
$\left(\begin{tabular}{ccc|c}
$\omega_{11}$ & 0 & $\omega_{13}$ &  1\\
$\omega_{21}$ & $\omega_{22}$ & $\omega_{23}$ & 1 \\
$\omega_{31}$ & 0 & $\omega_{33}$ & 1 \\
\hline
0 & 0 & 0 & 1
\end{tabular}\right)$}
\end{center}

\columnbreak

\begin{center}
	\scalebox{0.60}{
$\left(\begin{tabular}{ccc|c}
$\omega_{11}$ & $\omega_{12}$ & $\omega_{13}$ & 1\\
$\omega_{21}$ & 0 & 0 & 1\\
$\omega_{31}$ & $\omega_{32}$ & $\omega_{33}$ & 1\\
\hline
0 & 0 & 0 & 1\\
\end{tabular}\right)$}
\end{center}

\columnbreak

\begin{center}
	\scalebox{0.60}{
$\left(\begin{tabular}{ccc|c}
$\omega_{11}$ & $\omega_{12}$ & $\omega_{13}$ & 1\\
$\omega_{21}$ & 0 & $\omega_{23}$ & 1\\
$\omega_{31}$ & 0 & $\omega_{33}$ & 1\\
\hline
0 & 0 & 0 & 1\\
\end{tabular}\right)$}
\end{center}

\end{multicols}

\begin{multicols}{6}

\begin{center}
	\scalebox{0.60}{
$\left(\begin{tabular}{ccc|c}
0 & $\omega_{12}$ & $\omega_{13}$ & 1\\
$\omega_{21}$ & $\omega_{22}$ & $\omega_{23}$ & 1\\
$\omega_{31}$ & 0 & $\omega_{33}$ & 1\\
\hline
0 & 0 & 0 & 1\\
\end{tabular}\right)$}
\end{center}

\columnbreak

\begin{center}
	\scalebox{0.60}{
$\left(\begin{tabular}{ccc|c}
0 & $\omega_{12}$ & $\omega_{13}$ & 1\\
$\omega_{21}$ & 0 & $\omega_{23}$ & 1\\
$\omega_{31}$ & $\omega_{32}$ & $\omega_{33}$ & 1\\
\hline
0 & 0 & 0 & 1\\
\end{tabular}\right)$}
\end{center}

\columnbreak


\begin{center}
	\scalebox{0.60}{
$\left(\begin{tabular}{ccc|c}
$\omega_{11}$ & $\omega_{12}$ & $\omega_{13}$ & 1\\
$\omega_{21}$ & $\omega_{22}$ & $\omega_{23}$ & 1\\
$\omega_{31}$ & 0 & $\omega_{32}$ & 1\\
\hline
0 & 0 & 0 & 1\\
\end{tabular}\right)$}
\end{center}

\columnbreak

\begin{center}
	\scalebox{0.60}{
$\left(\begin{tabular}{ccc|c}
$\omega_{11}$ & $\omega_{12}$ & $\omega_{13}$ & 1\\
$\omega_{21}$ & $\omega_{22}$ & $\omega_{23}$ & 1\\
$\omega_{31}$ & $\omega_{32}$ & 0 & 1\\
\hline
0 & 0 & 0 & 1\\
\end{tabular}\right)$}
\end{center}

\columnbreak
\begin{center}
	\scalebox{0.60}{
$\left(\begin{tabular}{ccc|c}
$\omega_{11}$ & $\omega_{12}$ & $\omega_{13}$ & 1\\
$\omega_{21}$ & $\omega_{22}$ & $\omega_{23}$ & 1\\
$\omega_{31}$ & $\omega_{32}$ & $\omega_{33}$ & 1\\
\hline
0 & 0 & 0 & 1\\
\end{tabular}\right)$}
\end{center}

\end{multicols}


\end{document}